\documentclass[11pt,reqno]{amsart}

\usepackage{amssymb}
\usepackage{amscd}
\usepackage{amsfonts}
\usepackage{mathrsfs}
\usepackage{setspace}
\usepackage{version}

\numberwithin{equation}{section} \theoremstyle{plain}
\newtheorem{theorem}[subsection]{Theorem}
\newtheorem{proposition}[subsection]{Proposition}
\newtheorem{lemma}[subsection]{Lemma}
\newtheorem{corollary}[subsection]{Corollary}
\newtheorem{definition}[subsection]{Definition}

\newtheorem*{mainthm3-repeat}{Theorem \ref{mainthm3}}
\newtheorem*{gromov-cor-repeat}{Corollary \ref{gromov-cor}}
\newtheorem*{sum-product-rpt}{Theorem \ref{sum-product-fp}}
\newtheorem*{mainthm-repeat}{Theorem \ref{mainthm2}}
\newtheorem*{virtually-solvable-repeat}{Lemma \ref{virtually-solvable}}

\renewcommand{\leq}{\leqslant}
\renewcommand{\geq}{\geqslant}
\newsavebox{\proofbox}
\savebox{\proofbox}{\begin{picture}(7,7)  \put(0,0){\framebox(7,7){}}\end{picture}}

\newcommand\Z{\mathbb{Z}}
\newcommand\R{\mathbb{R}}

\newcommand\C{\mathbb{C}}
\newcommand\N{\mathbb{N}}
\newcommand\G{\mathbf{G}}
\newcommand\A{\mathbb{A}}
\newcommand\SL{\operatorname{SL}}
\newcommand\Upp{\operatorname{Upp}}

\newcommand\GL{\operatorname{GL}}
\newcommand\PGL{\operatorname{PGL}}
\newcommand\Rad{\operatorname{Rad}}

\newcommand\ess{\operatorname{ess}}
\newcommand\rk{\operatorname{rk}}

\newcommand\Hom{\operatorname{Hom}}

\newcommand\Inn{\operatorname{Inn}}
\newcommand\Out{\operatorname{Out}}
\newcommand\Aut{\operatorname{Aut}}
\renewcommand\P{\mathbb{P}}
\newcommand\F{\mathbb{F}}

\newcommand\eps{\varepsilon}
\newcommand\id{\operatorname{id}}

\def\endproof{\hfill{\usebox{\proofbox}}\vspace{11pt}}

\begin{document}

\title{Approximate subgroups of linear groups}
\author{Emmanuel Breuillard}
\address{Laboratoire de Math\'ematiques\\
B\^atiment 425, Universit\'e Paris Sud 11\\
91405 Orsay\\
FRANCE}
\email{emmanuel.breuillard@math.u-psud.fr}

\author{Ben Green}
\address{Centre for Mathematical Sciences\\
Wilberforce Road\\
Cambridge CB3 0WA\\
England }
\email{b.j.green@dpmms.cam.ac.uk}

\author{Terence Tao}
\address{Department of Mathematics, UCLA\\
405 Hilgard Ave\\
Los Angeles CA 90095\\
USA}
\email{tao@math.ucla.edu}


\begin{abstract}  We establish various results on the structure of approximate subgroups in linear groups such as $\SL_n(k)$ that were previously announced by the authors. For example, generalising a result of Helfgott (who handled the cases $n = 2$ and $3$), we show that any approximate subgroup of $\SL_n(\F_q)$ which generates the group must be either very small or else nearly all of $\SL_n(\F_q)$. The argument generalises to other absolutely almost simple connected (and non-commutative) algebraic groups $\G$ over a finite field $k$.  In a subsequent paper, we will give applications of this result to the expansion properties of Cayley graphs.
\end{abstract}

\maketitle
\tableofcontents

\section{Introduction}

The purpose of this paper is to study so-called \emph{approximate subgroups} of linear groups. In particular, we will study approximate subgroups of absolutely almost simple algebraic groups over an arbitrary field $k$, such as $\SL_n(k)$. (We review the definition and basic properties of absolutely almost simple algebraic groups in Section \ref{simple-sec}.)  These results were announced in \cite{bgst-announce}; closely related results have also been independently announced in \cite{pyber}.  We give an application of these results here to a conjecture of Babai and Seress \cite{babai}; in a subsequent paper \cite{bgt-expanders} we shall also apply these results to obtain expansion properties for various families of Cayley graphs.

We begin by recalling the notion of an approximate subgroup, first introduced (in the non-abelian setting) in \cite{tao-noncommutative}. See \cite{green-survey} for a more extensive motivating discussion.

\begin{definition}[Approximate subgroups] \label{def1.1} Let $K \geq 1$. A nonempty finite set $A$ in some ambient group $G$ is
called a \emph{$K$-approximate subgroup} of $G$ if
\begin{enumerate}
\item It is a \emph{symmetric subset of $G$}, by which we mean that if $a \in A$ then $a^{-1} \in A$, and that the identity lies in $A$;
\item There is a symmetric subset $X \subseteq G$ with $|X| \leq K$ such that $A \cdot A \subseteq X \cdot A$, where $A \cdot A = \{a_1 a_2 : a_1, a_2 \in A\}$ is the product set of $A$ with itself.
\end{enumerate}
We will refer to $K$-approximate subgroups informally as \emph{approximate groups} when the constant $K$ and the ambient group $G$ are either irrelevant to the discussion, or are clear from context.
\end{definition}

Note in particular that a $1$-approximate subgroup of $G$ is the same thing as a finite subgroup of $G$. For the rest of the paper we will assume that $K \geq 2$. In this regime there are $K$-approximate subgroups which are not necessarily close to genuine groups, the simplest example being that of a geometric progression $\{g^n : |n| \leq N\}$. There also exist higher-dimensional and nilpotent generalisations of this.  Again, \cite{green-survey} may be consulted for further discussion.

Many papers have been written in which the aim is to classify a certain class of approximate subgroups. For example, the Fre\u{\i}man-Ruzsa theorem \cite{ruz-freiman} provides a description of approximate subgroups of the integers $\Z$. ``Classification'' in this context must be interpreted quite loosely. The following notion of \emph{control}, first introduced in \cite{tao-solvable}, has proved useful in this context.

\begin{definition}[Control] \label{def1.2}Suppose that $A$ and $B$ are two sets in some ambient group,
and that $K \geq 1$ is a parameter. We say that $A$ is \emph{$K$-controlled by} $B$, or that $B$ \emph{$K$-controls} $A$, if $|B| \leq K|A|$ and there is some set $X$ in the ambient group with $|X | \leq K$
and such that $A \subseteq (X \cdot B) \cap (B \cdot X)$.
\end{definition}

Given this definition, one may describe the classification problem for approximate subgroups as follows: given some class $\mathscr{C}$ of approximate subgroups of a given group $G$, find some smaller, more highly-structured, class of approximate subgroups $\mathscr{C}'$ such that every object in $\mathscr{C}$ is efficiently controlled by an object in $\mathscr{C}'$.  In the next section we describe our results of this type in the case that the ambient group $G$ is an almost simple algebraic group, such as $\SL_d(k)$.
\vspace{11pt}

\noindent\textsc{notation.} The letter $C$ always denotes an absolute constant, but different instances of the notation may refer to different constants. If $C$ depends on some other parameter (for example, if we are working in $\SL_n$, $C$ might need to depend on $n$) then we will indicate this dependence with subscripts. If $A$ is a finite set then $|A|$ denotes the cardinality of $A$.
For non-negative quantities $X, Y$, we use $X \lesssim Y$ or $Y \gtrsim X$ to denote the estimate $X \leq K^C Y$, and $X \sim Y$ to denote the estimates $X \lesssim Y \lesssim X$. The symbol $p$ always denotes a prime number, and $\F_p$ denotes the field of order $p$. Finally, we use $A^n := \{ a_1 \ldots a_n: a_1,\ldots,a_n \in A\}$ to denote the $n$-fold product set of a collection $A$ of group elements, noting that if $A$ is a $K$-approximate subgroup then $|A^{n}| \leq K^{n-1}|A|$ for all positive integers $n$.

If $a, b \in G$, we write $b^a := a^{-1} ba$, and more generally if $B, A \subseteq G$, we write $B^a := \{ b^a: b \in B \}$, $b^A := \{ b^a: a \in A \}$, and $B^A := \{ b^a: a \in A, b \in B \}$.
\vspace{11pt}

\noindent\textsc{Acknowledgments.} EB is supported in part by the ERC starting grant 208091-GADA.
BG was, while this work was being carried out, a Fellow at the Radcliffe Institute at Harvard. He is very happy to thank the Institute for proving excellent working conditions.
TT is supported by a grant from the MacArthur Foundation, by NSF grant DMS-0649473, and by the NSF Waterman award.

The authors are particularly indebted to Tom Sanders for many useful discussions, and to Harald Helfgott for many useful discussions on product expansion estimates, and their relationship with the sum-product phenomenon. We also acknowledge the intellectual debt we owe to prior work of Helfgott, especially the paper \cite{helfgott-sl3}, and to the model-theoretic work of Hrushovski \cite{hrush}, without which we would not have started this project. We are also very grateful to Brian Conrad for some remarks on an earlier draft of this paper.\vspace{11pt}

Pyber and Szabo \cite{pyber} have independently announced a set of results which have significant overlap with those presented here, in particular establishing an alternate proof of Theorem \ref{mainthm2}, which also extends to cover all finite simple groups of Lie type.  There are some similarities in common in the argument (in particular, in the reliance on Lemma \ref{crucial}) but the arguments and results are slightly different in other respects.

\section{Statement of results}

In a celebrated paper \cite{helfgott-sl2}, H.~Helfgott provided a satisfactory solution to the classification problem for the group $\SL_2(\F_p)$. His methods adapt easily to (and in fact are rather easier in) $\SL_2(\C)$, a setting also studied by Chang \cite{chang}.  Helfgott's arguments give the following result.

\begin{theorem}[Helfgott]\label{helfgott-strong}  Suppose that $A \subseteq \SL_2(k)$ is a $K$-approximate subgroup of $SL_2(k)$.
\begin{enumerate}
\item If $k = \C$, then $A$ is $K^C$-controlled by $B$, an \emph{abelian} $K^C$-approximate subgroup of $\SL_2(k)$;
\item If $k = \F_p$, then $A$ is $K^C$-controlled either by a \emph{solvable} $K^C$-approxim- ate subgroup of $\SL_2(k)$ or by $\SL_2(k)$ itself.
\end{enumerate}
\end{theorem}

Helfgott's theorem has found many applications, for example to proving that certain Cayley graphs are expanders \cite{bourgain-gamburd} and to certain nonlinear sieving problems \cite{bgs}. For these applications, only the following somewhat weaker statement is necessary.

\begin{theorem}[Helfgott]\label{helfgott-weak}
Suppose that $A \subseteq \SL_2(\F_p)$ is a $K$-approximate subgroup that generates $\SL_2(\F_p)$. Then $A$ is $K^C$-controlled by either $\{\id\}$ or by $\SL_2(\F_p)$ itself.
\end{theorem}

Recently, but before our work, this result had been extended to fields $\F_q$ of prime power order $q=p^j$ by Dinai \cite{dinai}.  An extension to $\SL_n(\F_p)$ in the case of small $A$ (specifically, $|A| \leq p^{n+1-\delta}$ for some $\delta > 0$) had been obtained by Gill and Helfgott \cite{gill}.

The proof that such a statement suffices for the sieving work of Bourgain, Gamburd and Sarnak \cite{bgs} is contained in a very recent preprint of Varj\'u \cite{varju}; the original argument of \cite{bgs} required a careful analysis of the \emph{proof} of Helfgott's result.

Our first main result generalises Theorem \ref{helfgott-weak} as follows.

\begin{theorem}[Main theorem]\label{mainthm2}
Let $k$ be a finite field and let $\G$ be an absolutely almost simple algebraic group defined over $k$. Suppose that $A \subseteq \G(k)$ is a $K$-approximate subgroup that generates $\G(k)$. Then $A$ is $K^{C_{\dim(\G)}}$-controlled by either $\{\id\}$ or by $\G(k)$ itself.
\end{theorem}

\begin{remark} When we first announced this result in \cite{bgst-announce}, we restricted $\G$ to be an (implicitly $k$-split) Chevalley group.  Simultaneously with this announcement, Pyber and Szabo \cite{pyber} also announced Theorem \ref{mainthm2}, in the more general setting of arbitrary finite simple groups of Lie type.  Subsequently, while writing this paper, we realised that the argument did not use the $k$-split assumption anywhere in the proof, and in fact holds in the generality stated above.
\end{remark}

\begin{remark}  Note that the constant $C_{\dim(\G)}$ does not depend on the field $k$.
In particular, this theorem can be applied with $\G = \SL_d$ for any fixed $d$, and the constants now only depend on $d$.  In the case $\G(k) = \SL_3(\F_p)$ with $p$ prime, this result was established previously in \cite{helfgott-sl3}.  Our arguments share some features in common with those in \cite{helfgott-sl3}, most notably an emphasis on upper and lower bounds on the intersection of $A$ with maximal tori, and the use of a ``pivot'' argument inspired by proofs of the sum-product phenomenon.  \end{remark}

Theorem \ref{mainthm2} will be proven in Section \ref{simple-sec}.  In fact, we will prove a more precise result, Theorem \ref{mainthm-preciseform}, in which $A$ does not need to generate $\G(k)$, but merely has to be \emph{sufficiently Zariski dense} in the sense that it is not contained in a proper algebraic subgroup of $\G$ of bounded complexity (we will make these notions more precise later). Then the conclusion is that $A$ is controlled by either $\{\id\}$ or by the group $\langle A \rangle$ it generates. With this generalisation, the field $k$ can now be taken to be an infinite field, such as $\C$.

We will deduce Theorem \ref{mainthm2} from Theorem \ref{mainthm-preciseform} in Section \ref{simple-sec}.

As mentioned earlier one may use the arguments in \cite{bourgain-gamburd2} and the later paper \cite{varju} together with Theorem \ref{mainthm2} to generalise the aforementioned results on expanders and on the affine sieve. A particular application is to showing that the Suzuki groups (a family of simple groups constructed as subgroups of the symplectic groups $\mbox{Sp}_4(\F_q)$, $q = 2^{2n + 1}$) can be made into expanders. This result (established in our companion paper \cite{suzuki}) removes the lacuna in \cite{kassabov-nikolov-lubotzky}, which established that the family of all non-abelian simple groups \emph{other than} the Suzuki groups can be made into expanders. In another forthcoming paper \cite{bgt-expanders}, we will discuss in detail the application of our results and those of Pyber and Szabo to expansion in groups of Lie type in general.

By combining Theorem \ref{mainthm2} with standard results from noncommutative product set theory (see \cite{tao-noncommutative}) we obtain the following alternative formulation of our main theorem.

\begin{corollary}\label{helfgott-type-expansion}
Let $d \in \N$. Then there are $\epsilon(d)>0, C_d>0$ such that for every absolutely almost simple algebraic group $\G$ with $\dim(\G)\leq d$ defined over a finite field $k$, and every finite subset $A$ in $\G(k)$ generating $\G(k)$, and for all $0 < \epsilon < \epsilon(d)$, one of the following two statements holds:
\begin{enumerate}
\item[(i)] $|A| \gg_d |\G(k)|^{1-C_d\epsilon}$;
\item[(ii)] $|A^3| \geq |A|^{1+\epsilon}$.
\end{enumerate}
\end{corollary}

\begin{proof} If conclusion (ii) does not hold, then $(A \cup A^{-1} \cup \{\id\})^3$ is a $K^C$-approximate subgroup of $G:=\G(k)$ of cardinality at most $K^C |A|$ for some absolute constant $C>0$ and $K=\max(2,|A|^{\epsilon})$ (see \cite[Corollary 3.10]{tao-noncommutative}).

We are thus in a position to apply Theorem \ref{mainthm2}. If $\epsilon$ is small enough depending on $K^{C_{\dim(\G)}}$, it is not possible for $(A \cup A^{-1} \cup \{\id\})^3$ to be $K^{C_{\dim(\G)}}$-controlled by $\{\id\}$ for cardinality reasons unless $|A|=O_{\dim(\G)}(1)$, in which case (i) or (ii) is clear (for $\epsilon$ sufficiently small) just because the fact that $A$ generates $\G(k)$ implies that either $A=G$, or $|A^3| \geq |A|+1$.  Thus we need only consider the case when $A$ is controlled instead by $G=\G(k)$, in which case conclusion (i) easily follows.
\end{proof}

The fact that the above results extend to more general simple algebraic groups than $\SL_n$ conveys certain advantages. Over $\C$, for example, the general structure theory of algebraic groups implies, roughly speaking, that every closed connected (in the Zariski topology) subgroup of $\GL_n(\C)$ admits a quotient by a closed normal solvable subgroup which is an almost direct product of simple Lie groups over $\C$. By exploiting this theory, the generalisation of Theorem \ref{mainthm2} mentioned earlier, and the main result of \cite{bg-2}, we are able to establish the following result.

\begin{theorem}[Freiman-type theorem in $\GL_n$]\label{mainthm3}
Suppose that $A \subseteq \GL_n(k)$ is a $K$-approximate subgroup of $\GL_n(k)$, where $k$ is a field of characteristic zero. Then $A$ is $O_n(K^{O_n(1)})$-controlled by $B$, an $O_n(K^{O_n(1)})$-approximate group that generates a nilpotent group of nilpotency class \textup{(}i.e. step\textup{)} at most $n-1$.
\end{theorem}

This result is proven in Section \ref{freiman-sec2}. As a straightforward corollary we obtain the following special case of Gromov's theorem \cite{gromov}, first established by combining the Tits Alternative \cite{tits} with a theorem of Milnor \cite{milnor} and Wolf \cite{wolf}.

\begin{corollary}[Gromov's theorem for linear groups over $\C$]\label{gromov-cor}
Suppose that $k$ is a field of characteristic zero and that $G$ is a \textup{(}finitely-generated\textup{)} subgroup of $\GL_d(k)$ with polynomial growth. Then $G$ is virtually nilpotent.
\end{corollary}

We will recall the various notions mentioned here in \S \ref{freiman-sec2}, where the corollary is proved.

In the special case where $A$ is contained in $\SL_2(\C)$ or $\SL_3(\Z)$, Theorem \ref{mainthm3} was established by Chang \cite{chang}. The particular case of subsets of $\SL_2(\R)$ was established earlier by Elekes and Kir\'aly \cite{elekes-kiraly}. Note that, by the well-known fact that every finitely generated field of characteristic zero embeds in $\C$, the proof of Theorem \ref{mainthm3} is immediately reduced to the case $k=\C$. We are also able to say something about the structure of $K$-approximate subgroups of $\GL_n(k)$ where $k$ is a finite field, at least in the case $k = \F_p$, but it is not yet clear to us what the final form of such a result will be.  We hope to address these issues in a subsequent paper.

Qualitative forms of the above theorems follow from the work of Hrushovski \cite{hrush} (see in particular \cite[Theorem 1.3]{hrush} and \cite[Corollary 1.4]{hrush}). The main novelty of our work lies in the polynomial dependence on the approximation parameter $K$, which is absolutely essential for applications.

The dependence of constants on the dimension $n$ (or on $\dim(\G)$) in the above theorems are in principle explicitly computable. However, if one is willing to sacrifice such information, a key portion of the argument (the proof of the Larsen-Pink inequalities, described in \S \ref{larsen-pink-sec}) can be significantly simplified by the use of an ultrafilter argument; this is the approach we have chosen to take in this paper.

We remark that a key ingredient in Helfgott's work was the sum-product theorem in $\F_p$ \cite{bkt}. Our approach does not use this result (and neither does that of Pyber and Szabo). Moreover, we may in fact reverse the implication and deduce the sum-product theorem from the so-called Katz-Tao lemma \cite{katz} and Theorem \ref{mainthm2}.

\begin{theorem}[Sum-product theorem over $\F_p$]\label{sum-product-fp}
Let $p$ be a prime, and suppose that $A$ is a finite subset of $\F_p$ such that $|A\cdot A|, |A+ A| \leq K|A|$. Then either $|A| \leq K^C$ or $|A| \geq K^{-C}p$.
\end{theorem}

We give the deduction of Theorem \ref{sum-product-fp} from Theorem \ref{mainthm2} in Section \ref{sumproduct-sec}.
In principle one could also establish a very general form of the sum-product theorem, namely that every ``approximate subfield'' of some field $k$ is close to a genuine subfield of $k$. Such a result is stated for instance in \cite[Corollary 2.56]{tv-book}, and is essentially contained in \cite{bourgain-glibichuk-konyagin,bkt}. To obtain such a result by applying our arguments here would require the classification of maximal subgroups of $\SL_2(k)$, which is rather more complicated than the classification of subgroups of $\SL_2(\F_p)$. Therefore we do not give any such derivation here, leaving the details to the interested reader.

As remarked in the introduction, an immediate application of our results (noted in this context by Helfgott) is to cases of a conjecture of Babai and Seress on rapid generation in nonabelian finite simple groups. We state and prove these results in \S \ref{babai-sec}.

\section{Algebraic varieties}

In this section we review the definition and basic properties of algebraic varieties and algebraic groups. This material is needed in \S \ref{larsen-pink-sec} below. Readers may wish to refer to that section now for additional motivation.

In the classical presentation of the theory of algebraic geometry, varieties are defined at a qualitative level, with considerations such as the degree of the polynomials used to define the varieties being secondary to the main theory.  However, due to the quantitative nature of our analysis, we will need to present the foundations of algebraic geometry in a similarly quantitative manner, in particular assigning a ``complexity'' to each of the algebraic varieties that we encounter.  The precise definition of this complexity will not be important for us\footnote{If however one wished (say) to effectively control the constants $C_{\dim(G)}$ in Theorem \ref{mainthm2}, then one may wish to pay more attention to exactly how complexity is defined.}, but we will still need to fix one such definition of this concept in order to proceed with the rest of the proof.

We now quickly review the foundations of quantitative algebraic geometry, but with the proofs of various complexity bounds deferred to Appendix \ref{compact-sec}.  Our presentation here will be classical in nature, viewing algebraic varieties in terms of solutions of polynomials in affine (or projective) space, rather than by the modern machinery of schemes\footnote{Of course, it is quite likely one could recast the arguments here in scheme-theoretic language, although one might have to take care with issues of multiplicity when perfoming such tasks as counting the number of points of intersection $|A \cap V|$ between a finite set $A$ and a variety or scheme $V$.  One would also require a more intrinsic notion of complexity.}.  In particular, the basic notion of a \emph{complexity} of a variety will depend not just on the abstract variety itself, but also on how we choose to embed it into an affine or projective space.

Throughout this discussion $k$ will be an algebraically closed field.  We let $\A^n(k) = k^n$ be the standard $n$-dimensional affine space over $k$, and $\P^n(k) \equiv (k^{n+1} \backslash \{0\})/(k \backslash \{0\})$ the standard projective space.  Observe that $\P^n(k)$ can be covered by $n+1$ copies $\A_0^n(k),\ldots,\A_n^n(k)$, where $\A_i^n(k)$ is formed from $\P^n(k)$ by deleting the $i^{th}$ coordinate hyperplane.

\begin{definition}[Varieties]  Let $M \geq 1$ be an integer.
\begin{enumerate}
\item An \emph{affine variety} over $k$ of complexity at most $M$ is a subset $V \subseteq \A^n(k)$ of the form
$$ V = \{ x \in \A^n(k): P_1(x) = \ldots = P_m(x) = 0 \}$$
where $0 \leq n, m \leq M$, and $P_1,\ldots,P_m: \A^n(k) \to k$ are polynomials of degree at most $M$.

\item A \emph{projective variety} over $k$ of complexity at most $M$ is a subset $V \subset \P^n(k)$ of the form
$$ V = \{ x \in \P^n(k): P_1(x) = \ldots = P_m(x) = 0 \}$$
where $0 \leq n, m \leq M$, and $P_1,\ldots,P_m: k^{n+1} \to k$ are homogeneous polynomials of degree at most $M$.

\item A \emph{quasiprojective variety} over $k$ \textup{(}or \emph{variety} for short\textup{)} of complexity at most $M$ is a set of the form $V \backslash W \subset \P^n(k)$, where $V, W$ are projective varieties of complexity at most $M$.

\item A \emph{constructible set} over $k$ of complexity at most $M$ is a boolean combination of at most $M$ projective varieties of complexity at most $M$.
\end{enumerate}
We omit the phrase ``over $k$'' when the underlying field is clear from context.

The \emph{Zariski closure} of any subset of a constructible set $V$ is the intersection of $V$ with all the projective varieties containing that set.  We say that a set $A$ is \emph{Zariski dense} in a variety $V$ if its Zariski closure contains $V$.  A variety $W$ is \emph{closed} in another $V$ if its Zariski closure in $V$ is equal to $W$.

A variety is \emph{irreducible} if it cannot be expressed as the union of two proper closed varieties \textup{(}of any complexity\textup{)}.  In particular, an affine \textup{(}resp. projective\textup{)} variety is irreducible if it cannot be expressed as the union of two proper affine \textup{(}resp. projective\textup{)} varieties.
\end{definition}

\begin{remarks} An affine or projective variety of complexity at most $M$ is automatically a quasiprojective variety of complexity $O_M(1)$, and every quasiprojective variety of complexity $M$ is a constructible set of complexity $O_M(1)$.  Thus we may work without loss of generality with quasiprojective varieties or with constructible sets.  If we allow $M$ to be arbitrarily large, then these quantitative notions of a variety coincide with the usual ones in classical algebraic geometry.

The intersection or union of two affine varieties $V, W \subset \A^n(k)$ of complexity at most $M$ will be another affine variety of complexity $O_M(1)$.  Similarly, the Cartesian product of two affine varieties $V, W$ of complexity at most $M$ be an affine variety of complexity $O_M(1)$.  Analogous statements hold for projective and quasiprojective varieties (using the Segre embedding to model Cartesian products), and for constructible sets, with the exception that the union of quasiprojective varieties is merely a constructible set rather than a variety in general.

Our notion of complexity bounds the ambient dimension $n$, the maximum degree $d$ of the defining polynomials, and the number $m$ of such polynomials.  As it turns out, the third bound is essentially redundant, as it is controlled by the first two.  Indeed, the space of polynomials of degree at most $d$ that vanish on $V$ is a linear subspace of the space of all polynomials of degree at most $d$ on the affine space $\A^n(k)$, and thus has dimension $O_{d,n}(1)$.
\end{remarks}

We now give a useful result, that asserts that the Zariski closure operation is well-behaved with respect to complexity.

\begin{lemma}[Zariski closure preserves bounded complexity]\label{zar}  Let $V \subset W$ be constructible sets of complexity at most $M$.  Then the Zariski closure of $V$ in $W$ has complexity at most $O_M(1)$.
\end{lemma}

\begin{proof} See Appendix \ref{compact-sec}.
\end{proof}

Next, we recall the notion of a regular map between varieties, but with a notion of complexity attached.

\begin{definition}[Regular map]\label{regmap}  Let $V \subset \P^n(k)$ and $W \subset \P^m(k)$ be varieties, and let $M \geq 1$.  A map $f: V \to W$ is said to be \emph{regular} with complexity at most $M$ if $V, W$ are individually of complexity at most $M$, and if one can cover $V$ by some varieties $V_1,\ldots,V_r$ of complexity at most $M$ for some $r \leq M$ such that
\begin{enumerate}
\item for each $1 \leq j \leq r$, $f(V_j)$ is contained in an affine space $\A^m_{i_j}(k) \subset \P^m(k)$;
\item the map $f|_{V_j}$ has the form $(P_{j,1}/Q_{j,1},\dots, P_{j,m}/Q_{j,m})$, where the $P_{j,l}, Q_{j,l}$ are homogeneous polynomial maps from $k^{n+1}$ to $k$ with $\deg(P_{j,l}) = \deg(Q_{j,l}) \leq M$, and the $Q_{j,l}$ are non-vanishing on $V_j$.
\end{enumerate}
A regular map $\phi: V \to W$ is \emph{dominant} if $V$ is irreducible and $\phi(V)$ is Zariski-dense in $W$.
\end{definition}

Again, if we allow $M$ to be arbitrarily large, then this notion of a regular map coincides with the usual one in classical algebraic geometry.

The class of all varieties and the regular maps between them form a category.  This category is well-behaved with respect to complexity:

\begin{lemma}[Composition]\label{compos}
The composition of two regular maps of complexity at most $M$ is of complexity $O_M(1)$.  In particular, the restriction of a regular map of complexity at most $M$ to a subvariety of complexity at most $M$ is a regular map of complexity at most $O_M(1)$.

The image (or preimage) of a variety $V$ of complexity at most $M$ by a regular map of complexity at most $M$ is a constructible set of complexity $O_M(1)$.  In particular, by Lemma \ref{zar}, the Zariski closure of this image is a variety of complexity $O_M(1)$.
\end{lemma}

\begin{proof} See Appendix \ref{compact-sec}.
\end{proof}

Informally, the first part of Lemma \ref{compos} asserts that the class of \emph{bounded complexity} varieties and \emph{bounded complexity} regular maps also form a category.

To each non-empty (closed) affine variety $V \subset \A^n(k)$, we can assign a \emph{dimension} $\dim(V)$, which is an integer between $0$ and $n$, and can be defined as the maximal length $d$ of a proper nested sequence $\emptyset \subsetneq V_0 \subsetneq \ldots \subsetneq V_d \subset V$ of irreducible (closed) affine varieties in $V$.   The dimension of a projective variety is defined similarly.  The dimension of a quasiprojective variety or constructible set is then defined as the dimension of its Zariski closure.  We adopt the convention that the empty set has dimension $-1$.

We recall the following basic estimate.

\begin{lemma}[Weak Bezout]\label{bezout}  Let $V$ be a $0$-dimensional constructible set of complexity at most $M$.  Then $V$ is finite with cardinality $O_M(1)$.
\end{lemma}

\begin{proof} See Appendix \ref{compact-sec}.
\end{proof}

Recall that a property holds for \emph{generic} points in an variety\footnote{With this convention, if the variety is reducible, then any lower-dimensional components of the variety will not be considered generic.  Note that this is distinct from other conventional definitions of genericity in the literature.} if it holds at all points outside of a variety of strictly smaller dimension.  We introduce a quantitative notion of this concept:

\begin{definition}[Generic points]
If $M \geq 1$, and $V$ is a variety of complexity at most $M$, we say that a property $P(x)$ holds for \emph{$M$-generic points $x \in V$} if it holds for all points in $V$ outside of a subvariety $U$ of $V$ of complexity $M$ and dimension strictly less than that of $V$.  The set $V \backslash U$ will be called an \emph{$M$-Zariski open} subset of $V$.
\end{definition}

If $\phi: V \to W$ is a regular map, then it is well-known that $\phi(V)$ has dimension less than or equal to that of $V$.  Furthermore, the fibres $\phi^{-1}(x)$ generically have dimension $\dim(W)-\dim(V)$.  These facts behave well with respect to complexity, as the following lemma shows.

\begin{lemma}\label{dimlem}   Let $V, W$ be varieties of complexity at most $M$, and let $\phi: V \to W$ be a regular map of complexity at most $M$.  Then there exists an $O_M(1)$-Zariski open subset $V'$ of $V$, and a subvariety $W'$ of $W$ of dimension at most $\dim(V)$, with the following two properties:
\begin{enumerate}
\item \textup{(Generic mapping)} $\phi(V') \subset W'$.
\item \textup{(Generic fibres)}  For any $w \in W'$, the set $\{ v \in V': \phi(v)=w \}$ is a constructible set of complexity $O_M(1)$ and dimension at most $\dim(V) - \dim(W')$.
\end{enumerate}
We refer to $W'$ as an \emph{essential range} of $\phi$.

Finally, if $\phi: V \to W$ is a dominant map, then we may take $W'$ to be a $O_M(1)$-Zariski open subset of $W$.
\end{lemma}

\begin{proof} See Appendix \ref{compact-sec}.
\end{proof}

\begin{remark} If $V$ is irreducible, then we can automatically make $\phi$ dominant by replacing $W$ with the Zariski closure of $\phi(V)$ (using Lemma \ref{compos} to keep the complexity bounded).  On the other hand, if $V$ is not irreducible, then the (Zariski closure of the) essential range need not be unique.  Indeed, consider the example $V = (\{0\} \times k) \cup \{(1,0)\} \subset \A^2(k)$ with the projection map $\phi: \A^2(k) \to \A^1(k)$ defined by $\phi(x,y) := \phi(x)$.  Then both $\{0\}$ and $\{0, 1\}$ would qualify as essential ranges.
\end{remark}

\begin{lemma}[Slicing lemma]\label{slice-lem}   Let $V, W$ be varieties of complexity at most $M$, and let $S$ be a subvariety of $V \times W$ of complexity at most $M$ and dimension strictly less than $\dim(V) + \dim(W)$.  Then for $O_M(1)$-generic $v \in V$, the set $\{ w \in W: (v,w) \in S \}$ is a constructible set of complexity $O_M(1)$ and dimension strictly less than $\dim(W)$.
\end{lemma}

\begin{proof} See Appendix \ref{compact-sec}.
\end{proof}

\textsc{Algebraic groups.} We now recall the definition of an algebraic group, again with a notion of complexity attached to it.

\begin{definition}[Algebraic groups]\label{algdef}  Let $M\geq 1$.  An \emph{algebraic group} $G$ of complexity at most $M$ over an algebraically closed field $k$ is a variety $G$ of complexity at most $M$ which is also a group, with the group operations $\cdot: G \times G \to G$ and $()^{-1}: G\to G$ given by regular maps of complexity at most $M$.  If $G$ is an irreducible variety, we say that $G$ is a \emph{connected} algebraic group.
\end{definition}

\emph{Remark.} Recall that we are taking a classical algebraic geometry viewpoint here, so that our algebraic groups $G$ are not abstract, but instead come equipped with an embedding into affine or projective space.  This is necessary in order to make the notion of complexity well-defined.

\emph{Remark.} Our main theorem Theorem \ref{mainthm2} deals with \emph{linear} algebraic groups, that is algebraic groups whose underlying algebraic variety is affine. Nevertheless a fair amount of what we do in this paper, all of \S 4 and Appendix \ref{compact-sec}, and in particular the key Larsen-Pink inequality, hold for arbitrary algebraic groups including for instance abelian varieties.

\emph{Remark.} If a connected linear algebraic group is nilpotent or semisimple, and the underlying field has characteristic zero, then it is possible to show that after a change of variables (which may have unbounded complexity), one can find an isomorphic copy of this group whose complexity is bounded in terms of its dimension only.  However this is not true in general, and in particular fails for solvable groups over $\C$.  To see this, consider, for each $k \in \N$, the subgroup $G_k$ of $\GL_3(\C)$ defined by
\[ G_k := \left\{ \begin{pmatrix} x & 0 & 0  \\  0 & x^k & t \\ 0 & 0 & 1 \end{pmatrix} : x \in \C^{\times}, t \in \C \right\}.\]
One easily checks that $G_k$ is a connected solvable algebraic group of dimension $2$ whose center is isomorphic to the group of $k$-roots of unity. Lemma \ref{centralem} below shows that bounded complexity algebraic groups have a bounded complexity center. Thus the complexity of $G_k$ is not bounded as $k$ grows.  As this argument was purely algebraic, it also shows that any isomorphic copy of $G_k$ must also have complexity that is unbounded in $k$.

In the characteristic zero case, it turns out that such solvable examples, consisting entirely of upper-triangular examples, are essentially the \emph{only} way in which unbounded complexity of an algebraic subgroup can occur; we will formalise this observation (which greatly simplifies the proof of Theorem \ref{mainthm3}) in Section \ref{freiman-sec2}.



Given two varieties $V, W$ in $G$ of complexity at most $M$, the product set $V \cdot W$ is the image of $V \times W$ under the regular product map $\cdot: G \times G \to G$, and is thus a constructible set of complexity $O_M(1)$ by Lemma \ref{compos}. Furthermore, from Lemma \ref{dimlem} we can construct an essential range of the product map $\cdot: V \times W \to G$; we call such a range an \emph{essential product} of $V$ and $W$.

If $V$ is a subvariety of $G$, we observe from the invertibility of the conjugation map $g \mapsto a^{-1}ga$ in the category of regular maps that $V^a := \{ a^{-1}ga: g \in V \}$ is also a subvariety of $G$ of the same dimension.  From Lemma \ref{compos} we see that if $V, G$ have complexity at most $M$, then $V^a$ has complexity $O_M(1)$.  In a similar spirit, we have the following.

\begin{lemma}\label{centralem}  If $G$ is an algebraic group of complexity at most $M$, then for each $a \in G$, the
\emph{conjugacy class}
$$ a^G := \{ g^{-1} a g: g \in G \}$$
and the \emph{centraliser}
$$ Z(a) := \{ g \in G: ga = ag \}$$
are constructible sets of complexity $O_M(1)$.  Similarly, if $H$ is an algebraic subgroup of $G$ of complexity at most $M$, then the \emph{normaliser}
$$ N(H) := \{ g \in G: g^{-1} H g = H \}$$
and \emph{centraliser}
$$ Z(H) := \{ g \in G: gh=hg \hbox{ for all } h \in H \}$$
are algebraic subgroups of complexity at most $O_M(1)$.
\end{lemma}

\begin{proof} See Appendix \ref{compact-sec}.
\end{proof}

It is a standard fact that the Zariski closure of a group is an algebraic group.  The following lemma records a quantitative variant of this fact. It has been used in the past in several places in the literature, in particular in the work of Eskin-Mozes-Oh \cite{emo} on uniform exponential growth for linear groups (see also \cite{breuillard-gelander}). It was then put to use in additive combinatorics by Helfgott in \cite{helfgott-sl3} who called it ``escape from subvarieties".

\begin{lemma}[Escape from subvarieties]\label{zark}  For every $M$ there exists an integer $m \geq 1$ such that the following statement holds: for every algebraic group $G$ of complexity at most $M$, every subvariety $V$ of $G$ of complexity at most $M$, and every symmetric subset $A$ of $G$ containing $\id$ such that $A^m \subset V$, we have $A \subset H$ for some algebraic subgroup $H$ of $G$ contained in $V$ of complexity $O_M(1)$.
\end{lemma}

\begin{proof} See Appendix \ref{compact-sec}.
\end{proof}

The (almost) simplicity assumption in our main theorem, Theorem \ref{mainthm2}, will be used in a key way via the following lemma.

\begin{lemma}[Product-conjugation phenomenon for varieties]\label{escape}  Let $G$ be an algebraic group of complexity at most $M$ for some $M \geq 1$.  Let $V, W$ be algebraic varieties in $G$ of complexity at most $M$ such that
$$ 0 < \dim(V), \dim(W) < \dim(G).$$
Then at least one of the following holds:
\begin{enumerate}
\item \textup{(}$G$ is not sufficiently almost simple\textup{)} $G$ contains a proper normal algebraic subgroup $H$ of complexity $O_M(1)$ and positive dimension.
\item \textup{(}Product-conjugation phenomenon\textup{)} For $O_M(1)$-generic $a \in G$, there exists an essential product $(V^a \cdot W)^{\ess}$ of $V^a := a^{-1}Va$ and $W$ of dimension strictly greater than $\dim(W)$.
\end{enumerate}
\end{lemma}

\begin{proof} See Appendix \ref{compact-sec}.
\end{proof}

In practice, Lemma \ref{escape} only becomes useful when $G$ is almost simple (and in particular, when $G$ is connected).

We use the terminology ``product-conjugation phenomenon'' here in analogy with the ``sum-product phenomenon'' in fields.  The latter phenomenon asserts, roughly speaking, that the only way for a finite subset of a field to be approximately closed under both addition and multiplication is if it is essentially a subfield of the original field.  Similarly, the product-conjugation phenomenon refers to the heuristic that the only (algebraic) sets in an algebraic group that are approximately closed under both products and conjugation are normal algebraic subgroups.  This tension between the product operation and the conjugation operation is the key to our methods for controlling approximate groups, particularly in Lemma \ref{crucial}, in which the behaviour of the intersection of an approximate subgroup with a maximal torus is studied with respect to conjugation of that torus.

\section{The Larsen-Pink inequality for approximate subgroups}\label{larsen-pink-sec}

Suppose that $G$ is a simple algebraic group over some field and that $A \subseteq G$ is a $K$-approximate subgroup of $G$ which is ``sufficiently dense'' in $G$. The Larsen-Pink inequality asserts, roughly speaking, that one has a bound of shape $|A \cap V| \ll K^{O(1)} |A|^{\dim V/\dim G}$ for all bounded complexity subvarieties $V \leq G$. Here is a precise statement of this fact, which is fundamental to our work.

\begin{theorem}[Larsen-Pink inequality for symmetric subsets]\label{lpi}  Let $M \geq 1$, let $k$ be an algebraically closed field, and let $G$ be an algebraic group of over $k$ complexity at most $M$ and positive dimension.  Let $A$ be a symmetric subset of $G$.  Then at least one of the following statements hold:
\begin{enumerate}
\item[(i)] \textup{(}$G$ is not sufficiently almost simple\textup{)} $G$ contains a proper normal algebraic subgroup $H$ of complexity $O_M(1)$ and positive dimension;
\item[(ii)] \textup{(}$A$ is not sufficiently Zariski dense\textup{)} $A$ is contained in a subvariety of $G$ of complexity $O_M(1)$ and dimension strictly less than $G$;
\item[(iii)] For every subvariety $V$ of $G$ of complexity at most $M$, one has
\begin{equation}\label{av}
 |A \cap V| \leq C |A^C|^{\dim V/\dim G}
\end{equation}
for some $C = O_M(1)$.
\end{enumerate}
\end{theorem}

\emph{Remarks.} As the name suggests, the proof of this inequality will follow the arguments of Larsen and Pink \cite{larsen-pink}, who treated the case when $A$ was a genuine finite subgroup; but it turns out that the arguments extend without much difficulty to arbitrary  symmetric subsets.  The idea that the techniques of Larsen and Pink might be useful to us came to us through an analogous application of these ideas by Hrushovski \cite{hrush} (see also \cite{hrush-wagner}). It is also worth remarking that several special cases of Theorem \ref{lpi} are established in Helfgott's work \cite{helfgott-sl3} (for example in certain cases where $V$ is a maximal torus). The Larsen-Pink and Helfgott arguments are fundamentally rather similar, although this has only become clear rather recently.  A very similar inequality has also been recently established by Pyber and Szab\'o\cite{pyber2}.

\emph{Remark.} In this section we work with an arbitrary algebraic group $G$. In particular $G$ is not assumed to be linear and may for instance be an abelian variety.

We have stated the theorem for symmetric subsets $A$ in general, but will apply it when $A$ is an approximate subgroup (see Corollary \ref{lpi-cor} below).\vspace{11pt}

\textsc{Proof of Larsen-Pink.} The basic idea of the proof of Theorem \ref{lpi} is not especially difficult to describe\footnote{Our argument here is inspired by the presentation of the Larsen-Pink inequality in \cite[Proposition 5.5]{hrush}, though the notation and language used there is rather different from that given here.}, but the details require some effort. Suppose that $V^-, V^+$ is any pair of subvarieties of $G$ with $\dim(V^-) < \dim(V^+)$. Then we bound $|A \cap V^-|$ and $|A \cap V^+|$ in terms of $|A^{O(1)} \cap \tilde V^-|$ and $|A^{O(1)} \cap \tilde V^+|$, where the pair $\tilde V^-, \tilde V^+$ is somehow ``easier'' to deal with than $V^-, V^+$. By iterating this replacement algorithm or variants of it, we can reduce to the fairly trivial task of bounding $|A^{O(1)} \cap V|$ when $\dim(V) = 0$ or $\dim(G)$.

The idea behind the construction of $\tilde V^-$ and $\tilde V^+$ is not especially difficult to describe either. If $G$ is not sufficiently almost-simple then (i) holds. Otherwise, Lemma \ref{escape} applies and it roughly states that a generic $g \in G$ has the property that the dimension of the product $(V^-)^g \cdot V^+$ is strictly greater than $\dim(V^+)$. If no such $g$ lies in $A$ then $A$ cannot be sufficiently Zariski-dense, in which case (ii) holds. If there is some $g \in A$ with this property then set $\tilde V^+ := (V^-)^g \cdot V^+$ and take $\tilde V^-$ to be a suitable fibre of the product map
\[ (V^-)^g \times V^+ \rightarrow \tilde V^+.\]
Generically we will have $\dim(\tilde V^-) < \dim(V^-)$ and $\dim(\tilde V^+) > \dim(V^+)$, and this qualifies the pair $\tilde V^-, \tilde V^+$ as simpler than $V^-, V^+$ (in that this pair of varieties lies closer to a pair of varieties for which we can apply the trivial bound).
Furthermore it is clear than we can hope to bound $|A \cap V^-|$ and $|A \cap V^+|$ in terms of $|A \cap \tilde V^-|$ and $|A \cap \tilde V^+|$, as stated.

In reality the argument is slightly more complicated, this being due to our rather cavalier use of the word ``generically''. Furthermore the actual details of the iterative argument are tricker to handle than one might hope.

We begin by formulating a more precise lemma encapsulating the above observations.

\begin{lemma}[Inductive step]\label{ind-lp}
Suppose that $A$ is a symmetric subset of $G$ and that $V^-, V^+$ are subvarieties of $G$ of complexity at most $M$ with dimensions $d^- , d^+$ satisfying $0 < d^- \leq d^+ < \dim(G)$. Then one of the following three alternatives holds:
\begin{enumerate}
\item[(i)] \textup{(}$G$ is not sufficiently almost simple\textup{)} $G$ contains a proper normal algebraic subgroup $H$ of complexity $O_M(1)$ and positive dimension;
\item[(ii)] \textup{(}$A$ is not sufficiently Zariski dense\textup{)} $A$ is contained in a subvariety of $G$ of complexity $O_M(1)$ and dimension strictly less than $G$;
\item[(iii)] There are subvarieties $\tilde V_-$ and $\tilde V_+$ of complexity $O_M(1)$ such that
\begin{equation}\label{vv-bound} |A \cap V^-| |A \cap V^+| \leq 4|A^4 \cap \tilde V^-| |A^4 \cap \tilde V^+|.\end{equation}
and whose dimensions $\tilde d^-, \tilde d^+$ satisfy either
\begin{enumerate}
\item[(1)] $\tilde d^- < d^-$, $\tilde d^+ > d^+$ and $\tilde d^- + \tilde d^+ = d^- + d^+$ or
\item[(2)] $\tilde d^- + \tilde d^+ < d^- + d^+$.
\end{enumerate}
\end{enumerate}
\end{lemma}

We will prove this lemma below, but first we show how it implies Theorem \ref{lpi}.  The beef of such a deduction lies in a proper handling of the inequalities in (iii).

\emph{Proof of Theorem \ref{lpi} given Lemma \ref{ind-lp}.} Set $D := \dim(G)$. If $\dim(V) = 0$ then the result follows from Bezout's theorem (Lemma \ref{bezout}) whilst if $\dim(V) = \dim(G)$ then the result is trivial. We refer to these cases as the \emph{endpoint bounds}. Suppose, then, that $0 < \dim (V) < \dim(G)$. We shall inductively define pairs $(V^-_i, V^+_i)$ of varieties of complexities at most $M_i = O_{i,M}(1)$, $i = 0,1,2,\dots$, initialising so that $V^-_0, V^+_0 := V$. Suppose we are at the $i$th stage of the iteration: if the algorithm has not already stopped then the dimensions $d^{\pm}_i = \dim(V^{\pm}_i)$ will satisfy $0 < d_- \leq d_+ < \dim(G)$. This means that it is valid to apply Lemma \ref{ind-lp}, and we do this with $A^{4^i}$ in place of $A$ and $V^- := V^-_i$, $V^+ := V^+_i$. If case (i) or (ii) occurs then, provided the algorithm has only run for $O_{M,D}(1)$ steps, we obtain conclusions (i) and (ii) of Theorem \ref{lpi} respectively. In the more interesting case that (iii) occurs then we obtain varieties $\tilde V^-$ and $\tilde V^+$ such that
\eqref{vv-bound} holds and such that the dimensions $\tilde d^-, \tilde d^+$ satisfy (1) or (2) of Lemma \ref{ind-lp}.
 We then distinguish several cases, as follows.

\emph{Case 0.} If $\tilde d^- = 0$ and $\tilde d^+ = \dim(G)$ then \textsc{stop};

\emph{Case 1.} If $\tilde d^{-} = 0$ then set $V^-_{i+1} = V^-_i$ and $V^+_{i+1} = \tilde V^+$;

\emph{Case 2.} If $\tilde d^+ = \dim(G)$ then set $V^+_{i+1} = V^+_{i}$ and $V^-_{i+1} = \tilde V^-_i$;

\emph{Case 3.} If $0 < \tilde d^- \leq \tilde d^+ < \dim(G)$ and
\[ |A^{4^{i+1}} \cap \tilde V^+|^{1/\tilde d^+} \geq |A^{4^{i+1}} \cap \tilde V^-|^{1/\tilde d^-}\]  then set $V^-_{i+1} = V^-_i$ and $V^+_{i+1} = \tilde V^+$;

\emph{Case 4.} If $0 < \tilde d^- \leq \tilde d^+ < \dim(G)$ and
\[ |A^{4^{i+1}} \cap \tilde V^+|^{1/\tilde d^+} \leq |A^{4^{i+1}} \cap \tilde V^-|^{1/\tilde d^-},\]  then set $V^+_{i+1} = V^+_{i}$ and $V^-_{i+1} = \tilde V^-_i$.

By induction on $i$ one may check the lower bound
\begin{equation}\label{to-use} |A^{4^i} \cap V^{\pm}_i| \gg_{i,M} |A \cap V|^{d^{\pm}_i/\dim V},\end{equation} at each step.   To see this in cases 1 and 2, we use \eqref{vv-bound} (for $A^{4^i}$) and the endpoint bounds.  To see this in cases 3 and 4, we observe from \eqref{to-use} and the inequality $\tilde d^+ + \tilde d^- \leq d^+ + d^-$ that
\begin{align}\nonumber
 \max( |A^{4^{i+1}} \cap \tilde V^+_i|^{1/\tilde d^+}, & |A^{4^{i+1}} \cap \tilde V^-_i|^{1/\tilde d^-}) \\ &
\gg \max( |A^{4^{i+1}} \cap V^+_i|^{1/d^+}, |A^{4^{i}} \cap V^-_i|^{1/d^-})\label{sl}
\end{align}
from which the claim \eqref{to-use} follows.

If the algorithm we have described terminates in time $O_{D}(1)$, then Theorem \ref{lpi} now follows. Indeed upon termination we have $\dim(V^+_i) = \dim(G)$, and so the theorem is indeed an immediate consequence of \eqref{to-use}.

It is not immediately clear that the algorithm \emph{does} terminate, however.
To analyse this situation, note that by a similar inductive argument we may in fact establish the stronger bound
\begin{equation}\label{stronger} |A^{4^i} \cap V^{\pm}_i| \gg_{i,M} |A \cap V|^{(1 + c_D)^m d^{\pm}_i/\dim V}\end{equation} for some constant $c_D > 0$, where $m$ is the number of invocations of alternative (2) of Lemma \ref{ind-lp}.  Indeed, when (2) occurs thwn we have $\tilde d^+ + \tilde d^- < d^+ + d^-$, which allows us to raise the right-hand side of \eqref{sl} by $1+c_D$ for $c_D$ small enough, and the claim \eqref{stronger} then follows by repeating the proof of \eqref{to-use}.

If $m > 2\log D/\log(1 + c_D)$ then \eqref{stronger} already implies Theorem \ref{lpi} simply by applying the trivial bound $|A^{4^i} \cap V^{\pm}_i| \leq |A^{4^i}|$ to the left-hand side of \eqref{stronger}.  Suppose, then, that there are just $O_D(1)$ invocations of alternative (2) of Lemma \ref{ind-lp}. Observe that the algorithm cannot run for more than $D$ steps using only invocations of (1), since at any such invocation we have the inequality
\[ d^+_{i+1} - d^-_{i+1} > d^+_i - d^-_i.\] It follows that in this case the algorithm does indeed terminate in time $O_D(1)$, and so the proof of Theorem \ref{lpi} follows by our earlier remarks.\hfill $\Box$\vspace{11pt}

\emph{Remark.} A more careful inspection of the above argument shows that the constant $c_D$ can be bounded below effectively by $\gg D^{-2}$. This allows one to control the number of iterations here effectively, and this in turn means that in the bound $C_1|A^{C_2}|^{\dim V/\dim G}$ in Theorem \ref{lpi} one can take the exponent $C_2$ to be an effective quantity of the form $\exp(O((\dim G)^{O(1)}))$. Our arguments do not make the multiplicative factor $C_1$ effective, however, due to our reliance on ultrafilters. By laboriously replacing all of those arguments by effective algebraic geometry lemmas it ought to be possible to furnish an explicit dependence of $C_1$ on $\dim(G)$ and the complexity bound $M$, but this would be a considerable amount of work, the bounds would likely be bad and we do not at present have any applications for such a result.\vspace{11pt}

It remains, of course, to establish Lemma \ref{ind-lp}. Suppose then that we have subvarieties $V^-, V^+$ as in the statement of that lemma and that neither (i) nor (ii) of the lemma holds. That is to say, $G$ has no proper normal subgroups $H$ of complexity $O_M(1)$ and positive dimension, and $A$ is ``sufficiently Zariski-dense'' in the sense that it is not contained in any subvariety of complexity $O_M(1)$ and dimension strictly less than $\dim(G)$. We are free to choose these unspecified constants $O_M(1)$ as we please during the proof.

By Lemma \ref{escape}, we see that for a $O_M(1)$-generic set of points $a \in G$, there exists an essential product $(V^a_- \cdot V_+)^{\ess}$ of $V_-^a$ and $V_+$ of dimension strictly greater than $d'_+$.  Since $A$ is sufficiently Zariski-dense, at least one of these points lies in $A$. Henceforth we fix such an $a$ such that the essential product $W := (V^a_- \cdot V_+)^{\ess}$ of $V_-^a$ and $V_+$ of dimension $> d'_+$.

By definition of essential product, $W$ has complexity $O_M(1)$, and there exists a subvariety $S$ of $V_- \times V_+$ of dimension $< d_- + d_+ $ and complexity $O_M(1)$ such that
$$ (v^-)^a v^+ \in W$$
for all $(v^-,v^+) \in (V^- \times V^+) \backslash S$, and such that for any $w \in W$, the set
\begin{equation}\label{vav}
 \{ (v^-,v^+) \in V^- \times V^+: (v^-)^a v^+ = w \}
\end{equation}
is contained in a variety of complexity $O_M(1)$ and dimension at most
$$ \dim(V^- \times V^+) - \dim(W) < d'_- .$$
We distinguish two cases:

\emph{Case 1.} $|(A \times A) \cap S| \leq \frac{1}{2}|A \cap V^-||A \cap V^+|$;

\emph{Case 2.} $|(A \times A) \cap S| \geq \frac{1}{2}|A \cap V^-||A \cap V^+|$.

These will correspond to options (1) and (2) of alternative (iii) of Lemma \ref{ind-lp} respectively, as we shall now see.

Suppose first that we are in Case 1. For each $w \in W$, set
\[ F_w := \{ v^- \in A \cap V^- : (v^-)^a v^+ = w\},\]
By simple double-counting we have
\[ \frac{1}{2}|A \cap V^-| |A \cap V^+|  \leq |(A \times A) \cap S|  \leq \sum_{w \in W} |F_w|  \leq |F_{w_0}||A^4 \cap W|\]
where $w_0 \in W$ maximises the cardinality of $|F_{w_0}|$.  Applying Lemma \ref{compos} to \eqref{vav}, we see that $F_{w_0}$ is contained in a subvariety $\tilde V^-$ of $G$ with dimension $< d'_-$ and complexity $O_M(1)$; it is also clearly contained in $A^4$.  Setting $\tilde V^+ := W$, we obtain conclusion (1) of Lemma \ref{ind-lp} (iii) in this case.

Now suppose alternatively that we are in Case 2. By Lemma \ref{slice-lem}, we can find a subvariety $V^{\prime -}$ of $V^-$ of complexity $O_M(1)$ and dimension $< d^-$ such that for every $v^- \in V^- \backslash V^{\prime -}$ the fibre $F_{v^-} := \{ v^+ \in V^+: (v^-, v^+) \in S \}$ is contained in a variety of dimension  $< d^{ +}$ and complexity $O_M(1)$.  Write
\[ S_0 := \{(v^-, v^+) \in S : v^- \in V^{\prime -}\}.\] We divide into two further subcases:

\emph{Case 2a.} $|(A \times A) \cap S_0| \geq \frac{1}{2}|(A \times A) \cap S|$;

\emph{Case 2b.} $|(A \times A) \cap (S \setminus S_0)| \geq \frac{1}{2}|(A \times A) \cap S|$.

Recall that both are subcases of Case 2, which means that $|(A \times A) \cap S| \geq \frac{1}{2} |A \cap V^-||A \cap V^+|$.

In Case 2a, simply take $\tilde V^- := V^{\prime -}$ and $\tilde V^+ := V^+$ and note that
\[ |(A \times A) \cap S_0| \leq |A \cap V^{\prime -}| |A \cap V^+| = |A \cap \tilde V^-||A \cap \tilde V^+|.\]
This verifies alternative (2) of Lemma \ref{ind-lp} (iii) in this case.

In Case 2b, take $\tilde V^- := V^-$ and take $\tilde V^+$ to be that amongst the fibres $F_{v^-}$, $v^- \in V^- \backslash V^{\prime -}$, having largest intersection with $A$. Then

\[ |(A \times A) \cap S| \leq 2|(A \times A) \cap (S \setminus S_0)| \leq \sum_{v^- \in V^-} |F_{v^-} \cap A| \leq |A \cap \tilde V^-| |A \cap \tilde V^+|,\] confirming alternative (2) of Lemma \ref{ind-lp} (iii) in this case.

All eventualities having been covered, the proof of Lemma \ref{ind-lp} (and hence that of Theorem \ref{lpi}) is complete.\hfill $\Box$\vspace{11pt}

If $A$ is a $K$-approximate subgroup of a group $G$, then clearly we have $|A^m| \leq K^{m-1} |A|$ for any $m \in \N$.  As a consequence of Lemma \ref{zark} and Theorem \ref{lpi} (applied to $A^{m_0}$ for some large integer $m_0$), we obtain the following corollary.

\begin{corollary}[Larsen-Pink for approximate subgroups]\label{lpi-cor}  Let $K, m_0, M \geq 1$, let $k$ be an algebraically closed field, and let $G$ be an algebraic group of over $k$ of complexity at most $M$ and positive dimension.  Let $A$ be a $K$-approximate subgroup of $G$.  Then at least one of the following statements hold.
\begin{enumerate}
\item[(i)] \textup{($G$ is not sufficiently almost simple)} $G$ contains a proper normal algebraic subgroup of complexity $O_{M,m_0}(1)$ and positive dimension;
\item[(ii)] \textup{($A$ is not sufficiently Zariski dense)} $A$ is contained in an algebraic \emph{subgroup} of $G$ of complexity $O_{M,m_0}(1)$ and dimension strictly less than $G$;
\item[(iii)] For every $1 \leq m \leq m_0$ and every subvariety $V$ of $G$ of complexity at most $M$, one has
\begin{equation}\label{av-mult2}
 |A^m \cap V| \ll_{M,m_0} K^{O_{M,m_0}(1)} |A|^{\dim V/\dim G}.
\end{equation}
\end{enumerate}
\end{corollary}

We note the following important consequence of this corollary.

\begin{lemma}[Centralisers are rich]\label{centra-rich}  Let $G$ be an algebraic group of complexity at most $M$, let $A$ be a $K$-approximate subgroup in $G$, and let $a \in A$.  Then at least one of the following statements holds:
\begin{enumerate}
\item[(i)] \textup{($G$ is not sufficiently almost simple)} $G$ contains a proper normal algebraic subgroup $H$ of complexity $O_{M}(1)$ and positive dimension;
\item[(ii)] \textup{($A$ is not sufficiently Zariski dense)} $A$ is contained in an algebraic subgroup of $G$ of complexity $O_{M}(1)$ and dimension strictly less than $G$;
\item[(iii)] For every $a \in A$, one has
\begin{equation}\label{av-mult4}
 |A^2 \cap Z(a)| \sim_M  |A|^{\dim(Z(a))/\dim(G)},
\end{equation}
in the sense that the left and right sides are equal up to multiplication by a quantity of the form $O_M(K^{O_M(1)})$.
\end{enumerate}
\end{lemma}

\begin{proof}  The upper bound implicit in \eqref{av-mult4} comes from Corollary \ref{lpi-cor} and Lemma \ref{centralem}, so we focus on the lower bound.    We allow all implied constants to depend on $M$.

The fibres of the regular map $\phi: G \to G$ defined by $\phi(g) := g^{-1} a g$ are all cosets of $Z(a)$, and thus have dimension $\dim(Z(a))$.  As a consequence, $\phi(G) = a^G$ has dimension $\dim(G) - \dim(Z(a))$.  By Lemma \ref{compos}, this variety has complexity $O(1)$.  Applying Corollary \ref{lpi-cor}, we thus have either (i) or (ii), or else
$$ |\phi(A)| \ll K^{O(1)} |A|^{1-\dim(Z(a))/\dim(G)}.$$
By the pigeonhole principle, there thus exists $b \in a^G$ such that
$$ |\{ g \in A: g^{-1} a g = b \}| \gg  K^{-O(1)} |A|^{\dim(Z(a)) / \dim(G)}.$$
But if $g, h$ lie in the above set, then $hg^{-1}$ lies in $A^2 \cap Z(a)$.  This gives the desired lower bound for \eqref{av-mult4}.
\end{proof}

\emph{Remarks.} We remark on the connection between our proof of the Larsen-Pink inequalities and Helfgott's work in \cite{helfgott-sl3}. If $A$ is an approximate subgroup of $G = \SL_n(\F_p)$, Helfgott obtains a bound of the form $|A \cap T| \ll K^{O(1)} |A|^{1/(n+1)}$, where $T$ is any maximal torus in $G$; this is a special case of the upper bound in \eqref{av-mult4}. He does this by examining products $T^{a_1} \cdot \dots \cdot T^{a_{n+1}}$ of conjugates of $T$, showing that for tuples $(a_1,\dots, a_{n+1})$ in a Zariski-dense set the fibres of the map
\[ (t_1,\dots, t_{n+1}) \mapsto t_1^{a_1} \dots t_{n+1}^{a_{n+1}}\] have size $O_n(1)$. In a sense, this argument is precisely the same as our argument above in this special case, except that the necessary algebraic geometry is done in a completely explicit fashion (which has the advantage that effective bounds are produced).  It is very helpful in the arguments in \cite{helfgott-sl3} that $\dim T$ exactly divides $\dim G$, but this is probably not essential.
We remark that Helfgott also obtains lower bounds for the intersection of certain maximal tori with $A$, but not for all tori containing a regular semisimple element of $A$.

\section{The case of almost simple algebraic groups}\label{simple-sec}

We turn now to the main business of the paper and, in particular, the proof of Theorem \ref{mainthm2}.  We begin by reviewing some standard terminology concerning almost simple algebraic groups. In this section we deal with \emph{linear algebraic groups} only.

Let $k$ be an arbitrary field and $\overline{k}$ an algebraic closure of $k$. By definition, an \emph{absolutely almost simple $k$-algebraic group} $\G$ is a \emph{non-abelian} connected linear algebraic group defined over $k$ with no non-trivial proper normal connected algebraic subgroup defined over $\overline{k}$; in particular, $\G$ has no proper normal algebraic subgroups of positive dimension. When $k$ is algebraically closed we simply talk about \emph{almost simple algebraic groups}. We denote by $G=\G(k)$ its group of $k$-points.

From now on, unless otherwise stated, we will assume that $k$ is algebraically closed because our main result, i.e. Theorem \ref{mainthm-preciseform}, is formulated in this setting. We will consider non-algebraically closed fields only when applying Theorem  \ref{mainthm-preciseform} to the finite field case in order to prove Theorem \ref{mainthm2} from the Introduction.

When $k$ is algebraically closed (as we now assume), almost simple algebraic groups have a well-known structure and are parametrised by a pair $(\mathfrak{g},\Lambda)$, where $\mathfrak{g}$ is a finite dimensional simple complex Lie algebra (given in terms of its defining root system $\Delta$) and where $\Lambda$ is a certain free abelian group to be chosen among a finite collection of such. We refer the reader to Humphreys' book \cite{humphreys} for a pleasant and thorough introduction to this material.

Finite dimensional simple complex Lie algebras are parametrised by (reduced, irreducible) root systems. According to the Cartan-Killing classification, these root systems fall into $4$ infinite families $(A_n)_{n \geq 1}, (B_n)_{n \geq 2}, (C_n)_{n \geq 3}$ and $(D_n)_{n \geq 4}$ or are one of the $5$ exceptional ones $E_6$, $E_7$, $E_8$, $F_4$ and $G_2$. The infinite families correspond to the so-called \emph{classical Lie algebras} $\mathfrak{sl_{n+1}}, \mathfrak{so_{2n+1}}, \mathfrak{sp_{2n}}$ and $\mathfrak{so_{2n}}$.

To every (abstract, reduced, irreducible) root system $\Delta$ are attached two free abelian groups of the same common finite rank, the lattice of roots $\Lambda_r$ and the lattice of weights $\Lambda_w$ which contains it. Their rank $\rk$ is called the \emph{rank} of the root system, or of $\mathfrak{g}$. It can be seen that the group $\Lambda_w/\Lambda_r$ has cardinality bounded by $\rk+1$.


To every pair $(\mathfrak{g},\Lambda)$ one may associate in a unique fashion an almost simple algebraic group $\G$ defined over $k$  associated to $(\mathfrak{g},\Lambda)$. Its center can be identified with $\Hom(\Lambda/\Lambda_r,k^{*})$, hence is of size at most $\rk(\G)+1$ (equality being realized for instance when $\G=\SL_n$ and $k=\C$). When $\Lambda=\Lambda_r$, $\G$ is said to be \emph{adjoint}, while when $\Lambda=\Lambda_w$, $\G$ is said to be \emph{simply connected}. Note that adjoint groups $\G(k)$ are center-free, and hence abstractly simple. For the explicit construction of $\G$ from $(\mathfrak{g},\Lambda)$, see \cite{humphreys}.

It follows in particular from this construction that $\dim \G = \dim \mathfrak{g}$, which in turn implies that there are only finitely many isomorphism classes of almost simple algebraic groups over $k$ whose dimension is given, because the same holds for complex Lie algebras as a result of the Cartan-Killing classification.

Furthermore, the complex Lie algebra $\mathfrak{g}$ admits a free abelian subgroup of full rank, denoted by $\mathfrak{g}_{\Z}$, such that $\mathfrak{g}_\Z$ is stable under braket, and $\mathfrak{g}_{\Z}\otimes k$ coincides with the Lie algebra $\mathfrak{g}_{k}$ of $\G(k)$ defined as the Zariski tangent space at the identity of $\G(k)$. The group $\G(k)$ acts naturally on $\mathfrak{g}_{k}$ via the \emph{adjoint representation}.

\emph{Example.} If $\mathfrak{g}$ is a complex simple Lie algebra of type $A_n$, then $\mathfrak{g}$ is isomorphic to $\mathfrak{sl}_{n+1}(\C)$, the Lie algebra of traceless complex square matrices of size $n+1$. Let $\mathfrak{h}$ be the subalgebra of diagonal matrices. The root system $\Delta$ of $\mathfrak{g}$ consists of $\frac{n(n-1)}{2}$ roots of $\mathfrak{g}$ defined as the linear forms $(\lambda_i - \lambda_j)_{i \neq j}$, where $\lambda_i \in \mathfrak{h}^*$ is the $i$-th diagonal coefficient. The lattice $\Lambda_r$ is the subgroup of $\mathfrak{h}^*$ generated by the $(\lambda_i - \lambda_j)_{i \neq j}$, while the lattice of weights $\Lambda_w$ is the subgroup of $\mathfrak{h}^*$ generated by the $(\lambda_i)_i$'s. Here $\Lambda_w/\Lambda_r \simeq \Z/(n+1)\Z$ and $\rk = n$. When $\Lambda=\Lambda_w$, then the group associated to $(\mathfrak{g},\Lambda)$ and an algebraically closed field $k$ is exactly the classical group $\SL_{n+1}(k)$ of determinant $1$ matrices over the field $k$. When $\Lambda=\Lambda_r$ the corresponding group is $\PGL_{n+1}(k)$ of projective linear transformations of $k^{n+1}$.

A \emph{torus} $S$ in $\G$ is a connected diagonalisable algebraic subgroup of $\G$. By \emph{diagonalisable}, we mean that under some faithful algebraic embedding of $\G$ in $\GL_n$ (recall that every linear algebraic group admits a faithful algebraic finite dimensional representation), $S$ is a subgroup of the diagonal matrices for some choice of basis. This notion is well-defined (see the books of Humphreys \cite{humphreys} or Borel \cite{borel}). A \emph{maximal torus} is a torus of maximal dimension in $\G$. It can be shown that they are all conjugate in $\G(k)$.

Given a maximal torus $T$ of $\G$, the group of characters $X(T):=\Hom(T,k^{*})$ can be identified with the lattice subgroup $\Lambda$. Roots $\alpha \in \Delta$ thus give rise to characters.

The centraliser of $T$ coincides with $T$, while $T$ has finite index in its normaliser $N(T)$. The quotient group $N(T)/T$ is the so-called Weyl group of $\G$. It permutes the roots of $\G$ and can be identified with the abstract Weyl group of the root system $\Delta$.

The subgroups $T_\alpha = \ker \alpha$  are the so-called \emph{maximal singular tori}; they are subtori of $T$. An element $g$ in $\G(k)$ is said to be \emph{semisimple} if it is contained in some torus. A semisimple element of $T$ is called \emph{regular semisimple} if it is not contained in any of the $T_\alpha$, $\alpha \in \Delta$. Regular semisimple elements are precisely those elements $g$ of $\G(k)$ such that the multiplicity of the eigenvalue $1$ in the matrix representation $\operatorname{Ad}(g)$ on $\mathfrak{g}_{k}$ is minimal. This is clearly a Zariski-open condition. The centraliser $Z(a)$ of a regular semisimple element $a$ has minimal possible dimension, namely the rank $\rk \G$. Its connected component\footnote{We thank J.-P.~Serre for pointing out to us that $Z(a)$ may not be connected if $\G$ is not simply connected.} of the identity is the unique maximal torus $T=Z(a)_0$ containing $a$. In particular $Z(a) \leq N(T)$ and $|Z(a)/Z(a)_0| \leq |N(T)/T|$.

If $k$ is not algebraically closed, then $\G(\overline{k})$ is of the type described above, but there are in general several isomorphism classes of $k$-algebraic groups which are nonetheless isomorphic over $\overline{k}$. These are called the $k$-forms of $\G$. When $k$ is finite, they have been entirely classified by Steinberg \cite{steinberg2} and Tits \cite{tits2, tits3}.

For the above and more background on almost simple algebraic groups and Chevalley groups, we refer the reader to the books by Humphreys \cite{humphreys} and Borel \cite{borel} and to Steinberg's lecture notes \cite{steinberg}. We record some of the above facts in the form of a lemma.

\begin{lemma}\label{maxtor}  Let $M \geq 1$, and let $\G$ be an almost simple\footnote{Recall that we require almost simple groups to be non-abelian, thus excluding the degenerate one-dimensional case.} algebraic group over an algebraically closed field $k$ of complexity bounded by $M$.
\begin{enumerate}
\item Regular semisimple elements of $G=\G(k)$ form a Zariski open subset of $\G$ whose complement is of complexity at most $O_M(1)$.
\item Each regular semisimple element $a$ is contained in precisely one maximal torus, which is the connected component of the centraliser $Z(a)$ of $a$.
\item All maximal tori are conjugate, and their common dimension is the rank $\rk(\G)$. We have $\rk(\G) < \dim(\G)$.
\item Given a maximal torus $T$, the \emph{Weyl group} $N(T)/T$ has cardinality $O_M(1)$.
\end{enumerate}
\end{lemma}

\begin{proof}  Most of these results are standard, except perhaps for the complexity bound in (i). This follows from the fact that the adjoint representation of $\G(k)$ on $\mathfrak{g}_k$ has complexity $O_M(1)$ and the condition that $\mbox{Ad}(g)$ has the maximum number of distinct eigenvalues is clearly given by the non-vanishing of polynomials of degree $O_{\dim(\mathfrak{g})}(1)$ in the entries of $\mbox{Ad}(g)$.
\end{proof}

\emph{Example.}  If $\G = \SL_n$, then the semisimple elements are the diagonalisable matrices in $\G(k)$, and the regular semisimple elements are those matrices whose eigenvalues are distinct.  All maximal tori are conjugate to the group of diagonal matrices in $\SL_n$, and the Weyl group is isomorphic to the permutation group $S_n$, which has order cardinality $n!$.  We recommend to readers who are unfamiliar with the general theory of algebraic groups that they use this model case $\G = \SL_n$ as a running example.

\emph{Remark.} The above classification of almost simple algebraic groups over an algebraically closed field shows that \emph{up to isomorphism} there are only finitely many such groups in each given dimension. In particular there always exists a model of $\G$ whose complexity is bounded in terms of $\dim(\G)$ only. The same applies to semisimple algebraic groups (see \cite{humphreys}), they are isomorphic to a direct product of simple algebraic groups modulo a finite central kernel (whose size is bounded in terms of $\dim(\G)$ only). In characteristic zero, the situation is even better: there are only finitely many conjugacy classes of almost simple algebraic subgroups of $\GL_d$, and their complexity is thus $O_d(1)$ (this follows easily from complete reducibility and the classification of irreducible linear representations of simple algebraic groups via their highest weight). However things seem more delicate in positive characteristic.

With these preliminaries out of the way, we now come to a key definition.

\begin{definition}[Involved torus]
If $A$ is a symmetric subset of $\G(k)$, and $T$ is a maximal torus of $\G$, we say that $T$ is \emph{involved} with $A$ if $A^2 \cap T$ contains a regular semisimple element.  We let $\mathscr{T} = \mathscr{T}(A)$ denote the collection of all involved tori.
\end{definition}

The following lemma is fundamental to our work.  It was inspired by the proofs of the sum-product phenomenon by Bourgain, Glibichuk, and Konyagin \cite{bourgain-glibichuk-konyagin} and its interpretation given by Helfgott in \cite{helfgott-sl3} in terms of group actions. It has also been discovered independently by Pyber and Szabo \cite{pyber}. We thank Harald Helfgott for pointing out a simplification to the proof of this lemma which, in the form below, is closer to the argument of Pyber and Szabo than it is to our original one. Since neither argument is particularly long, we give the slightly simpler version here, referring the reader to our announcement \cite{bgst-announce} for the original.

\begin{lemma}[Crucial lemma]\label{crucial}  Let $M, K \geq 1$, and let $\G$ be an almost simple algebraic group with complexity at most $M$ over an algebraically closed field $k$. Let $A$ be a $K$-approximate subgroup of $\G$.  Then at least one of the following statements holds.
\begin{itemize}
\item[(i)] \textup{(}$A$ is not sufficiently Zariski dense\textup{)} $A$ is contained in an algebraic subgroup of $\G$ of complexity $O_{M}(1)$ and dimension strictly less than $\G$.
\item[(ii)] \textup{(}$A$ is small\textup{)} $|A| \ll_M K^{O_M(1)}$.
\item[(iii)] The set $\mathscr{T}$ of involved tori $T$ has the cardinality bounds
\begin{equation}\label{tcard}
 K^{-O_M(1)} |A|^{1 - \frac{\dim(T)}{\dim(\G)}} \ll_M |\mathscr{T}| \ll_M K^{O_M(1)} |A|^{1 - \frac{\dim(T)}{\dim(\G)}}
 \end{equation}
and is invariant under conjugation by $A$ \textup{(}and hence by the group $\langle A \rangle$ generated by $A$\textup{)}.  In other words, if $T$ is involved, then so is $a^{-1} T a$ for any $a \in A$.
\end{itemize}
\end{lemma}

\begin{proof}  We allow all implied constants to depend on $M$.

Let us first show the cardinality bounds \eqref{tcard}.  By Lemma \ref{maxtor} and Corollary \ref{lpi-cor}, either (i) holds, or the number of elements of $A^2$ which are not regular semisimple is
$O( K^{O(1)} |A|^{1-\frac{1}{\dim(\G)}} )$ (note that the first option in Corollary \ref{lpi-cor} cannot occur because $\G$ is almost simple.)  Thus, either (i) or (ii) holds, or the number of regular semisimple elements of $A^2$ is comparable to $|A^2| = K^{O(1)} |A|$.  By Lemma \ref{maxtor}, these regular semisimple elements can be partitioned into involved tori, and by Lemma \ref{centra-rich}, either (i) holds, or each involved torus $T$ contains a number of elements of $A^2$ comparable to $K^{O(1)} |A|^{\dim(T)/\dim(\G)}$.  By Lemma \ref{maxtor} and Corollary \ref{lpi-cor}, either (i) or (ii) holds, or the number of elements in $A^2 \cap T$ that are not regular semisimple is at most $K^{O(1)} |A|^{(\dim(T)-1)/\dim(\G)}$.  Thus, either (i) or (ii) holds, or every involved torus $T$ contains a number of \emph{regular semisimple elements} in $A^2$ comparable to $|A|^{1 - \frac{\dim(T)}{\dim(\G)}}$.  The claim \eqref{tcard} follows.

Now, let $T$ be an involved torus, let $a \in A$, and let $\tilde T := a^{-1} T a$ be the conjugate of $T$ by $a$.  Being involved, $T$ is the connected component of the centraliser $Z(b)$ of some regular semisimple $b \in A^2$.  Applying Lemma \ref{centra-rich}, we either have (i), or
$$ |A^2 \cap Z(b)| \gg K^{-O(1)} |A|^{\dim(T) / \dim(\G)}.$$
Suppose we are in the latter case. Since $|Z(b)/T| \leq |N(T)/T| = O(1)$ one coset of $T$ in $Z(b)$ must have a large intersection with $A^2$, in particular $$ |A^4 \cap T| \gg K^{-O(1)} |A|^{\dim(T) / \dim(\G)}.$$ Conjugating by $a$, we conclude
$$ |A^6 \cap \tilde T| \gg K^{-O(1)} |A|^{\dim(T) / \dim(\G)}.$$
On the other hand, we have
$$ |A \cdot (A^6 \cap \tilde T)| \leq |A^7| \ll K^{O(1)} |A|.$$
Thus, by the pigeonhole principle, we can find an element $g$ of $A \cdot (A^6 \cap \tilde T)$ which has $\gg K^{-O(1)} |A|^{\dim(T)/\dim(\G)}$ representations of the form $g = a t$, where $a \in A$ and $t \in A^6 \cap \tilde T$.  But if $at = a't'$ for some $a, a' \in A$ and $t, t' \in A^6 \cap \tilde T$, then $(a')^{-1} a = t' t^{-1} \in A^2 \cap \tilde T$.  Thus we have
$$ |A^2 \cap \tilde T| \gg K^{-O(1)} |A|^{\dim(T)/\dim(\G)}.$$
On the other hand, by Lemma \ref{maxtor} and Corollary \ref{lpi-cor}, the number of elements of $A^2 \cap \tilde T$ which are not regular semisimple is at most \[ O( K^{O(1)} |A|^{(\dim(T)-1)/\dim(\G)}).\]  Thus we are either in case (ii) or case (iii), and the claim follows.
\end{proof}

To proceed further, we need to divide into two cases depending on whether the group $\langle A \rangle$ generated by $A$ is finite or infinite.  In the infinite case, we use the following\footnote{In the case $k=\C$, one could also use Jordan's lemma here.} simple lemma.

\begin{lemma}\label{cover}  Let $M \geq 1$, and let $\G$ be an almost simple algebraic group of complexity at most $M$ over an algebraically closed field $k$.  Let $H$ be an infinite subgroup of $\G(k)$, which can be covered by finitely many left cosets $aT$ of a torus $T$ of complexity at most $M$ inside $G$.  Then there exists a torus $S$ \textup{(}not necessarily maximal\textup{)} of complexity $O_M(1)$ such that $H \cap S$ is infinite, and $H$ is contained in the normaliser $N(S)$ of $S$.
\end{lemma}

\begin{proof}  It is clear from the assumption that $H \cap T$ is clearly infinite, so if $H$ is contained in $N(T)$ then we are done. Suppose, then, that there exists $h \in H$ such that $\tilde T := h^{-1} T h \neq T$.  Conjugating $H$ by $h$, we see that $H$ can be covered by finitely many left cosets of $\tilde T$.  The intersection of a left-coset of $T$ and a left-coset of $\tilde T$ is either empty or a left-coset of the subgroup $T \cap \tilde T$, which has complexity $O_M(1)$ and has strictly smaller dimension. On the other hand, Lemma \ref{quivar} guarantees that the connected component $S$ of $T \cap \tilde T$ has complexity $O_M(1)$. We thus see that $H$ can be covered by finitely many left cosets of the torus $S$.  As we cannot have an infinite descent of tori of decreasing dimension, the claim follows by iterating this argument.
\end{proof}

We may now state and prove a result that is, essentially, a precise form of our main theorem. Sometimes (for example for applications to expanders) this precise form is more helpful than Theorem \ref{mainthm2} itself. We give the deduction of Theorem \ref{mainthm2} afterwards.

\begin{theorem}\label{mainthm-preciseform}  Let $M, K \geq 1$, and let $\G$ be an almost simple algebraic group of complexity at most $M$ over an algebraically closed field $k$, and let $A$ be a $K$-approximate subgroup of $\G(k)$.  Then at least one of the following statements hold.
\begin{itemize}
\item[(i)] \textup{(}$A$ is not sufficiently Zariski dense\textup{)} $A$ is contained in an algebraic subgroup of $\G$ of complexity $O_{M}(1)$ and dimension strictly less than $\G$.
\item[(ii)] \textup{(}$A$ is small\textup{)} $|A| \ll_M K^{O_M(1)}$.
\item[(iii)] \textup{(}$A$ controlled by $\langle A \rangle$\textup{)} The group $\langle A \rangle$ generated by $A$ is finite, and has cardinality $|\langle A \rangle| \ll_M K^{O_M(1)} |A|$.
\end{itemize}
\end{theorem}

\begin{proof}  We allow all constants to depend on $M$.  From Lemma \ref{maxtor} (i) we see in particular that $\G$ has complexity $O_M(1)$.  There are two cases, depending on whether $\langle A \rangle$ is finite or not.

First suppose that $\langle A \rangle$ is finite.  From Lemma \ref{crucial}, either (i) or (ii) holds, or the set $\mathscr{T}$ of involved tori is invariant under conjugation by $\langle A \rangle$.  On the other hand,
$\langle A \rangle$ is certainly a $1$-approximate subgroup of $G=\G(k)$, so by Corollary \ref{lpi-cor}, either (i) holds, or the intersection of $\langle A \rangle$ with any maximal torus has cardinality $\ll |\langle A \rangle|^{\dim(T)/\dim(\G)}$.  Since $|N(T)/T| \ll 1$ by Lemma \ref{maxtor}, we conclude that the stabiliser of the action of $\langle A \rangle$  on any torus in $\mathscr{T}$ has cardinality $\ll |\langle A \rangle|^{\dim(T)/\dim(\G)}$.  By the orbit-stabiliser theorem, this implies that
$$ |\mathscr{T}| \gg |\langle A \rangle|^{1 - \frac{\dim(T)}{\dim(\G)}}.$$
Comparing this with \eqref{tcard} we obtain (iii).

Now suppose instead that $\langle A \rangle$ is infinite.  The orbit-stabiliser theorem and Lemma \ref{crucial} then implies that either (i) or (ii) holds, or a finite index subgroup of $\langle A \rangle$ lies in $N(T)$ for some $T \in \mathscr{T}$.  In the latter case, since $N(T)/T$ is finite, this implies that $\langle A \rangle$ is covered by a finite number of cosets of a torus $T$.  By Lemma \ref{cover}, we thus see that $\langle A \rangle$ is contained in the normaliser $N(S)$ of a torus $S$ of dimension at least one and complexity $O_M(1)$.  But such normalisers have complexity $O_M(1)$ by Lemma \ref{centralem}, and dimension strictly less than $\dim(\G)$ by Lemma \ref{maxtor}, so we obtain (i) in this case.
\end{proof}

We can now prove Theorem \ref{mainthm2}, whose statement we first recall.

\begin{mainthm-repeat}[Main theorem]
Let $k$ be a finite field and let $\G$ be an absolutely almost simple algebraic group defined over $k$, and suppose that $A \subseteq \G(k)$ is a $K$-approximate group that generates $\G(k)$. Then $A$ is $K^{C_{\dim(\G)}}$-controlled by either $\{\id\}$ or by $\G(k)$ itself.
\end{mainthm-repeat}

\begin{proof}
Let $\overline{k}$ be the algebraic closure of $k$. The algebraic group $\G$ is a $k$-form of the almost simple algebraic group $\G(\overline{k})$. Such $k$-forms over finite fields have been entirely classified by Steinberg \cite{steinberg2} and Tits \cite{tits2, tits3}. In particular, for every type of Dynkin diagram there is a bounded number (independently of $k$) of such forms.  As the statement of the theorem is invariant with respect to group isomorphism, we may thus assume that the complexity of $\G$ is bounded in terms of $\dim(\G)$ only.

Applying Theorem \ref{mainthm-preciseform}, it suffices to establish that option (i) of that theorem cannot hold for $|k| \geq C_{\dim(\G)}$. In other words, we must establish that $\G(k)$ is ``sufficiently-Zariski-dense'' in $\G(\overline{k})$ in the sense that the lowest complexity $M_k$ of any proper subvariety of $\G(\overline{k})$ containing $\G(k)$ tends to infinity with $|k|$. Suppose, then, that $V \subseteq \G(\overline{k})$ is a proper subvariety of complexity $M$ (not necessarily defined over $k$). The required statement then follows by combining the following two facts:
\begin{itemize} \item $V(k) = V \cap \G(k)$ has cardinality at most $O_{M}(|k|^{\dim V})$;
\item (Lang-Weil, \cite{lang-weil}) The cardinality of $\G(k)$ is $\gg_M |k|^{\dim \G}$.
\end{itemize}
In connection with the first fact, a well-known variant of the Schwarz-Zippel lemma (see \cite[Lemma A.3]{ell-ob-tao} for a simple proof) asserts that an irreducible affine variety $V$ defined by $m$ equations of degrees $\leq d$ and coefficients in $\overline{k}$ has at most $d^m (|k| + 1)^{\dim V}$ points over $k$. In connection with the second, we note that in all cases (both split or non-split, see e.g. \cite{carter}) an exact formula for $|\G(\F_q)|$ in terms of $q$ is known, and so the somewhat deep Lang-Weil estimate is really overkill here. For example, it is extremely well-known and easy to see that
\[ |\SL_n(\F_q)| = \frac{(q^n - 1) (q^n - q)(q^n - q^2) \dots (q^n - q^{n-1})}{q-1}.\]

\noindent This concludes the proof of Theorem \ref{mainthm2}.\end{proof}

\section{Approximate subgroups of $\GL_n(\C)$} \label{freiman-sec2}

The objective of this section is to establish Theorem \ref{mainthm3}, characterising approximate subgroups of linear groups in characteristic zero.  As remarked after the statement of that theorem, we may assume without loss of generality that the ambient field is the complex numbers, and so we may rephrase the statement there with the following one.

\begin{mainthm3-repeat}  Let $K, n \geq 1$, and suppose that $A \subseteq \GL_n(\C)$ is a $K$-approximate subgroup. Then $A$ is $O(K^{O_n(1)})$-controlled by $B$, a $CK^{C_n}$-approximate subgroup of $\GL_n(\C)$ that generates a nilpotent group of step at most $n-1$.
\end{mainthm3-repeat}

There is an appetising strategy for establishing this result, and to outline it we recall some basic definitions and facts from algebraic group theory. In particular let us recall that if $G$ is an algebraic group over some algebraically closed field $k$ then the \emph{radical} $\Rad(G)$ is the maximal normal connected solvable algebraic subgroup of $G$. An algebraic group with trivial radical is said to be \emph{semisimple}; the quotient $G/\Rad(G)$ has this property.

\begin{proposition}[Semisimple algebraic groups over $\C$]\label{classification}
Suppose that $G$ is a connected semisimple linear algebraic group over $\C$. Then $G$ is an almost direct product of almost simple algebraic groups $G_1,\dots,G_k$, each arising from one of the irreducible root systems mentioned earlier. More precisely, there is a map $\pi : G_1 \times \dots \times G_k \rightarrow G$ with $|\ker \pi| = O_{\dim(G)}(1)$. Note also that \textup{(}as a consequence of the classification of irreducible root systems\textup{)} each $G_i$ may be realised as a linear algebraic group of complexity $O_{\dim(G)}(1)$.
\end{proposition}

\begin{proof} See \cite[Proposition 14.10]{borel}.  Note that this proposition in fact establishes some additional facts, for instance one can take the $G_i$ to be the minimal nontrivial (smooth) connected normal (closed) $k$-subgroups of $G$ that pairwise commute, and one can also take the kernel $\ker \pi$ to be central.  We will not need this additional structure here.
\end{proof}

As a consequence of this classification and of Theorem \ref{mainthm-preciseform} it is not hard to prove the following result.

\begin{proposition}\label{semisimple-c}
Suppose that $G$ is a connected semisimple linear algebraic group over $\C$ and that $A \subseteq G$ is a $K$-approximate group generating a Zariski-dense subgroup of $G$. Then $|A| = O(K^{O_{\dim(G)}(1)})$.
\end{proposition}
\begin{proof} Decompose $G$ as an almost direct product $\pi(G_1 \times \dots \times G_k)$ of almost simple linear algebraic groups over $\C$ as in Proposition \ref{classification}, the kernel of $\pi$ having size $t = O_{n}(1)$. An easy exercise confirms that the pullback $B := \pi^{-1}(A)$ is a $(Kt)$-approximate subgroup of $G_1 \times \dots \times G_m$ which generates a Zariski-dense subgroup. Each of the projections $\pi_i(B)$ to the factors $G_i$ is then a $(Kt)$-approximate subgroup of $G_i$ generating a Zariski-dense subgroup of $G_i$. The complexities of the groups $G_i$ are bounded uniformly by $O_{\dim(G)}(1)$, and so we may apply Theorem \ref{mainthm-preciseform}. Clearly neither option (i) nor option (iii) of that theorem can occur, so we are forced to conclude that (ii) holds, or in other words that $|\pi_i(B)| = O(K^{O_{\dim(G)}(1)})$. This immediately implies the stated bound on $|A|$.
\end{proof}

One is now tempted to argue as follows. Suppose that $A$ is an approximate subgroup of $\GL_n(\C)$. Let $G$ be the Zariski closure of $\langle A \rangle$, and consider the image of $A$ under the projection to $G/\Rad(G)$. By the above remarks this is small and so a large piece of $A$ is contained in a coset of the solvable group $\Rad(G)$. It is then a relatively easy matter to apply the main result of \cite{bg-2} to show that this piece in turn is controlled by a nilpotent approximate group.

The only problem with this argument is that Proposition \ref{classification} applies only to \emph{connected} semisimple groups. Although a linear algebraic group has only finite many connected components, there is no absolute bound on their number in terms of the underlying dimension $n$. It is nonetheless possible to make this sketch of the argument work, but to do so we were forced to introduce a notion of  ``bounded complexity Zariski-closure'' so as to control the number of connected components of $G/\Rad(G)$. The details were somewhat complicated. However it turns out that by resting a little more heavily on known facts from the theory of linear algebraic groups we can use a somewhat different argument and thereby sidestep this issue. Critical in this regard is the following result, which makes up for the fact that we cannot bound the number of connected components of a virtually solvable group, by bounding instead the number of \emph{solvable} components.

\begin{lemma}[Number of solvable components]\label{virtually-solvable}
Any virtually solvable subgroup of $\GL_n(\C)$ \textup{(}that is to say, any subgroup of $\GL_n(\C)$ with a solvable subgroup of bounded index\textup{)} contains a normal solvable subgroup of index $O_n(1)$.
\end{lemma}

\begin{proof}
In fact, any such subgroup contains a subgroup of index $O_n(1)$ which can be conjugated inside the group $\Upp_n(\C)$ of upper-triangular matrices. This lemma is ``well-known'' and follows from results of Malcev \cite[\S 3.6]{wehrfritz} and Platonov \cite[\S 10.11]{wehrfritz}. We offer a self-contained sketch in Appendix \ref{solv-app}.
\end{proof}

We will also require the following bound on the outer automorphism group of a semisimple group.

\begin{lemma}[Bound on $\Out(H)$]\label{aut}
Suppose that $H$ is a connected semisimple complex algebraic group. Then the group $\Out(H) := \Aut(H)/\Inn(H)$ has size $O_{\dim H}(1)$.
\end{lemma}
\begin{proof}
See \cite[IV.14.9]{borel}.
\end{proof}

We turn now to the details of the proof of Theorem \ref{mainthm3}. Henceforth, all implied constants to depend on $n$.
Let $\Gamma := \langle A \rangle$ be the subgroup of $\GL_n(\C)$ generated by $A$, and let $G$ be the Zariski closure of $\Gamma$. Let $G_0$ be the connected component of the identity of $G$, thus $G_0$ is a connected normal subgroup of $G$ of finite (but potentially large) index.  Let $H := G_0/\Rad(G_0)$ be the quotient of $G_0$ by its radical; $H$ is then a connected semisimple algebraic group whose structure is as described in Proposition \ref{classification}. We distinguish two cases, depending on whether $H$ is trivial or not.

\emph{Case 1:} $H$ is trivial. This implies that $\Gamma$ is virtually solvable. In this case, we can apply Lemma \ref{virtually-solvable} and infer that $\Gamma$ has a normal solvable (indeed upper triangular) subgroup $\Gamma_0$ of index $O(1)$. By the pigeonhole principle some coset of $\Gamma_0$ contains at least $|A|/C$ points of $A$, where $C = O(1)$, and thus $|A^2 \cap \Gamma_0| \geq \frac{1}{C}|A|$. It follows from basic multiplicative combinatorics \cite{tao-noncommutative}, or more specifically \cite[Proposition 2.1 (iv)]{bg-2}, that $B := (A^2 \cap \Gamma_0)^3$ is a $K^{O(1)}$-approximate group that $K^{O(1)}$-controls $A^2 \cap \Gamma_0$. Furthermore, $A$ is itself $m$-controlled by $A^2 \cap \Gamma_0$, where $m = O(1)$ is the number of cosets of $\Gamma_0$ containing an element of $A$. It follows that $B$ is a solvable $K^{O(1)}$-approximate subgroup of $\GL_n(\C)$ which $K^{O(1)}$-controls $A$. Applying the main result of \cite{bg-2}, Theorem \ref{mainthm3} follows immediately.

\emph{Case 2:} $H$ is non-trivial. Write $R := \Rad(G_0)$ for brevity. Now $R$ is a characteristic subgroup of $G_0$, since the image of $R$ under any automorphism of $G_0$ is also a connected normal algebraic solvable subgroup of $G_0$.  Noting that $G$ normalises $G_0$, it follows that $G$ normalises $R$, and thus $G$ also acts by conjugation on the quotient group $H=G_0/R$. This conjugation action induces a homomorphism $\pi: G \rightarrow \Aut(H)$. Note that $\ker(\pi)R/R=Z_{G/R}(H)$ is finite, since its intersection with the finite index subgroup $G_0$ of $G$ lies in the center $Z(H)$ of $H$, which is finite as $H$ is semisimple. In particular, $\ker \pi$ is virtually solvable. By essentially the same argument we used in Case 1, it suffices to show that some coset of $\ker \pi$ contains at least $K^{-O(1)}|A|$ points of $A$. To achieve this, it is obviously sufficient to show that $|\pi(A)| \leq K^{O(1)}$.

To do this, observe that $\pi(G)$ certainly contains the inner automorphisms $\Inn(H) = \pi(G_0)$. Recall also Lemma \ref{aut}, which asserts that the index $[\Aut(H) : \Inn(H)]$ is at most $O(1)$. Setting $\Gamma_0 = \Gamma \cap \pi^{-1}(\Inn(H))$, we conclude that $\Gamma_0$ is a subgroup of index $d = O(1)$ in $\Gamma$ such that $\pi(\Gamma_0)$ is Zariski dense in $\Inn(H)$. Now $\Inn(H)$ is a (connected) semisimple algebraic group, and by Lemma \ref{finite-index} the set $B:=A^{2d-1} \cap \Gamma_0$ generates $\Gamma_0$. Moreover, by the same multiplicative combinatorics results mentioned above, $B^3$ is an $O(K^{O(1)})$-approximate subgroup which $O(K^{O(1)})$-controls $A$, and therefore $\pi(B^3)$ is a Zariski-dense $O(K^{(O(1)})$-approximate subgroup of $\Inn(H)$ which $O(K^{O(1)})$-controls $\pi(A)$. It follows immediately from Proposition \ref{semisimple-c} that $|\pi(B^3)| = O(K^{O(1)})$, and hence that $|\pi(A)| = O(K^{O(1)})$ as required.\hfill $\Box$\vspace{11pt}

We turn now to the deduction of Corollary \ref{gromov-cor}, whose statement we recall below. Recall that a finitely generated group $G$ is said to have \emph{polynomial growth} (with exponent $d$) if it is generated by a symmetric set $S$ which satisfies some bound of the form
\begin{equation}\label{poly}
|S^r| \leq C_S r^d
\end{equation}
for all $r \in \N$, where $C_S$ is a constant which may depend on $S$. It is an easy exercise to show that this a group property, that is to say it does not depend on the choice of generators. Recall also that $G$ is \emph{virtually nilpotent} if it has a nilpotent subgroup of finite index.

\begin{gromov-cor-repeat}
Suppose that $k$ is a field of characteristic zero and that $G \leq \GL_n(k)$ is a \textup{(}finitely-generated\textup{)} subgroup of $G$ with polynomial growth. Then $G$ is virtually nilpotent.
\end{gromov-cor-repeat}
\begin{proof}
Let $S$ be a symmetric generating set containing $\id$. There are clearly arbitrarily large $r$ for which
\begin{equation}\label{good-r-bound} |S^{7r}| \leq 8^d |S^r|,\end{equation} since if not the polynomial growth hypothesis would be violated. Call these values good, and suppose in what follows that $r$ is good. By \cite[Corollary 3.10]{tao-noncommutative} the set $A := S^{3r}$ is a $K$-approximate group for some $K = O(1)^d$. By Theorem \ref{mainthm3}, there is some nilpotent group $H \leq \GL_n(k)$ and a coset $Hx$ such that $|A \cap Hx| \geq c_{n,d}|A|$, where $c_{n,d} > 0$ depends only on $n$ and $d$. We therefore have
\begin{equation}\label{s6-bd} |S^{6r} \cap H| = |A^2 \cap H| \geq c_{n,d}|A| \geq c_{n,d} |S^r|.\end{equation}
Replacing $H$ by the subgroup generated by $A^2 \cap H$ (if necessary) we may assume without loss of generality that $H \leq G$. Assume that $[G : H] = \infty$. Then, since $S$ generates $G$, it is easy to see that $S^m$ meets at least $m$ different right cosets of $H$, for every integer $m \geq 1$. It follows from this observation and \eqref{s6-bd} that
\[ |S^{6r + m}| \geq m c_{n,d} |S^r|.\]
Choosing $m > 8^d/c_{n,d}$ and some good value of $r$ with $r > m$, we obtain a contradiction to \eqref{good-r-bound}. Thus we were wrong to assume that $[G : H] = \infty$, and this concludes the proof.
\end{proof}

The above argument does not give effective bounds on $[G : H]$ (in terms of $n, C_s$ and $d$) on account of the ineffectivity in Theorem \ref{mainthm3}. To conclude this section we offer now a very brief sketch of how an explicit bound could be obtained without effectivising our main theorem. First, note that an immediate consequence of Corollary \ref{gromov-cor} and Lemma \ref{virtually-solvable} is that any group $G \leq \GL_n(k)$ with polynomial growth has a \emph{solvable} subgroup $H$ with index $[G : H] = O_n(1)$, where this $O_n(1)$ can be taken to be some explicit function of the form $\exp(n^C)$ by working through the proof of Lemma \ref{virtually-solvable} in Appendix \ref{solv-app}. By Lemma \ref{finite-index}, $H$ is finitely-generated and has polynomial growth (with effectively computable constant and growth exponent). Finally, one may run the argument used to prove Corollary \ref{gromov-cor} again, but now with $H$ in place of $G$, and with the main result of \cite{bg-2} (which is completely effective) in place of Theorem \ref{mainthm3}. We leave the details to the interested reader.


Finally let us remark that an inspection of the argument used to prove Corollary \ref{gromov-cor} shows that we do not need the full strength of the polynomial growth hypothesis \eqref{poly}; indeed we only need this hypothesis for a \emph{single} (sufficiently large) value of $r$. This type of strengthened version of Gromov's theorem first appeared in \cite{tao-solvable} in the solvable case and \cite{shalom} for arbitrary finitely generated groups. For non-virtually solvable linear groups it is also a immediate consequence of the uniform Tits alternative proved in \cite{breuillard-tits}.

\section{A conjecture of Babai and Seress}\label{babai-sec}

Some applications of Helfgott's product theorem regarding the diameter of the Cayley graphs of $\SL_2(\F_p)$, or equivalently the speed of generation of $\SL_2(\F_p)$, are mentioned in Helfgott's original paper \cite{helfgott-sl2}. Not surprisingly, our result provides similar corollaries. For example, we get a special case of a conjecture of Babai and Seress (\cite[Conjecture 1.7]{babai}).

\begin{theorem}[Diameter of $\G(k)$]\label{worst-case}
For every integer $d \in \N$ there is a constant $C = C_d>0$ such that the following holds. If $\G$ is an absolutely almost simple algebraic group with $\dim(\G) \leq d$ defined over a finite field $k$ and if $S$ is a set of generators for $G:=\G(k)$, then every element of $G$ is a word of length at most $C\log^{C}|G|$ in the elements of $S$.
\end{theorem}

\emph{Remark.} In this section, when we refer to a word of length at most $L$ in elements $x_1,\dots,x_n$, we refer to some expression of the form $x_{i_1}^{\eps_1}x_{i_2}^{\eps_2}\dots x_{i_L}^{\eps_L}$, where each $\eps_j$ lies in the set $\{-1,0,1\}$.

\emph{Proof.} Let $\delta := d/10\dim(\G)$. Using Corollary \ref{helfgott-type-expansion}, we examine the powers $S^3, S^9, S^{27},\dots$. Unless
\begin{equation}\label{big} |S^{3^n}| \geq |G|^{1-\delta},\end{equation}
that result implies that we have the growth bound
\[ |S^{3^{n+1}}| \geq |S^{3^n}|^{1 + \eps},\]
where $\eps = \eps_d > 0$. It follows that if \eqref{big} does not hold then
\[ |S^{3^n}| \geq 2^{(1+\eps)^n},\] and so
\[ N := 3^n \leq C_d \log^{C_d} |G|.\]
\newcommand\ad{\operatorname{ad}}
We have shown that words of length $N \leq C_d\log^{C_d}|G|$ in the generating set $S$ cover a set $B$ ($= S^N$) of size $|G|^{1-\delta}$. For such large subsets of $G$, arguments of Gowers \cite{gowers} may be used to complete the proof. Gowers proved that if $H$ is a group in which the smallest dimension of a non-trivial representation (over $\C$) is $d_{\min}$ then $A^3 = H$ whenever $|A| > |G|/d_{\min}^{1/3}$. See also \cite{babai-nikolov-pyber,nikolov-pyber}. To use Gowers' result, we note that by a theorem of Landazuri and Seitz \cite{landazuri-seitz} the smallest dimension of a non-trival \emph{projective} complex representation is at least $c_r |H|^{\frac{r}{\dim(\G)}} \geq c_r |H|^{10\delta}$ for every Chevalley group\footnote{Landazuri and Seitz state their result for groups $\G_{\ad}(\F_q)^+$ with $\G_{\ad}$ of adjoint type, but as they point out the general case follows immediately by Schur's lemma because $\G(\F_q)^+/Z(\G(\F_q)^+)=\G_{\ad}(\F_q)^+$.} $H=\G(k)^+$ of rank $r \leq d$ over $k$. It is well known that $\G(k)$ contains the Chevalley group $\G(k)^+$ as a subgroup of index $O_r(1)$ (see e.g. \cite{nori} 3.6(v)). We can thus apply the result to $\G(k)^+$ instead since $|B^2 \cap \G(k)^+| \gg_r |G|^{1-\delta}$ and conclude by Lemma \ref{finite-index} that $S^{O_r(1)}B^6=G$. \hfill $\Box$\vspace{11pt}


Babai and Seress actually conjectured that the constant $C_d$ and the implied multiplicative constant in the above theorem can be taken to be independent of the dimension, and that this bound should also hold for the alternating groups. Furthermore it is likely that the above theorem is true with the smallest possible exponent $C_d$, that is $C_d=1$. See \cite{breuillard-gamburd} for some evidence supporting this belief, which is related to the so-called \emph{Lubotzky-Weiss independence problem} for expanders \cite{lubotzky-weiss}.  Below we show that one can take $C_d=1$ for a \emph{random choice} of generating set $S$.

\begin{theorem}[Logarithmic diameter for random generators]\label{random-case}
For every integer $d \in \N$ there is a constant $C_d>0$ such that the following holds. If $G=\G(k)$ is the group of $k$-points of an absolutely almost simple algebraic group $\G$ of dimension at most $d$ over a finite field $k$, and if two elements $a,b$ are chosen at random from $G$, then with probability $1 - o_{|G|\rightarrow \infty}(1)$ every element $g \in G$ is a word of length at most $C_d \log |G|$ in $a,b, a^{-1}$ and $b^{-1}$.
\end{theorem}

We remark that the special case $G=\SL_2(\F_p)$, $p$ prime, of this theorem was established by Helfgott\cite{helfgott-sl2}.

To prove Theorem \ref{random-case} we will need the following result from \cite{gamburdHSSV}.

\begin{theorem}\label{GHSSV}\cite{gamburdHSSV} For every integer $d \in \N$ there is a constant $c_d>0$ such that the following holds. If $\G$ is an absolutely almost simple algebraic group of dimension at most $d$  over a finite field $k$, and if two elements $a,b$ are chosen at random from $G:=\G(k)$, then with probability $1 - o_{|G|\rightarrow \infty}(1)$, no nontrivial word $w$ of length $\leq c_d \log |G|$ in the free group $F_2$ evaluates to the identity on $a,b$.\end{theorem}
\emph{Remarks.} Equivalently, the corresponding Cayley graph of $G$ on generating set $\{a,b,a^{-1},b^{-1}\}$ has \emph{girth} at least $c_d \log |G|$. The constant $c_d$ can be taken to be any number smaller than $(\dim(\G) \log 3)^{-1}$.\vspace{11pt}

\emph{Proof of Theorem \ref{random-case}.} Set $N_1 := c_d\log |G|$. From Theorem \ref{GHSSV} it follows that if $a,b \in G$ are selected at random then, with probability $1 - o_{|G|\rightarrow \infty}(1)$, we have $|S^{N_1}| \geq 3^{c_d\log |G|} =: |G|^{\eta}$, say. Applying Corollary \ref{helfgott-type-expansion} $O_d(1)$ times, we see that $|S^{N_2}| \geq |G|^{1-\delta}$ for some $N_2 \ll_r N_1$, where $\delta := d/10\dim(\G)$ is the same quantity as in the proof of Theorem \ref{worst-case}. Applying the result of Gowers exactly as before we obtain $S^{3N_2} = G$, thereby confirming the result. \endproof\vspace{5pt}

Other results in a similar vein are available.  For example if $a$ and $b$ are fixed elements of $\SL_n(\Z)$ generating a Zariski-dense subgroup then we may look at these $a$ and $b$ when reduced to lie in $\SL_n(\F_p)$. By the strong-approximation result of Matthews-Weisfeiler-Vasserstein \cite{MVW} $a$ and $b$ generate $\SL_n(\F_p)$ for $p$ sufficiently large, and applying the Tits alternative \cite{tits} one may show that there are $\gg p^{\eta}$ distinct words $w(a,b)$ of length $c \log p$ (here, $\eta$ and $c$ may depend on $a$ and $b$ as well as on $n$). Proceeding exactly as above, one may then confirm that every element of $\SL_n(\F_p)$ is a word of length at most $C\log p$ in these generators $a,b$. We leave the details of this verification to the reader.

\section{The sum-product theorem over $\F_p$}\label{sumproduct-sec}

The aim of this section is to show how the sum-product theorem over $\F_p$ follows from our main results. We begin by recalling the statement of the sum-product theorem, which was first established in \cite{bkt} for ``reasonably large'' sets and then in \cite{bourgain-glibichuk-konyagin} in full generality.

\begin{sum-product-rpt}[Sum-product theorem over $\F_p$]
Let $p$ be a prime, and suppose that $A$ is a finite subset of $\F_p$ such that $|A\cdot A|, |A+ A| \leq K|A|$. Then either $|A| \leq K^C$ or $|A| \geq K^{-C}p$.
\end{sum-product-rpt}

We will need the so-called ``Katz-Tao lemma'' \cite{bkt,katz}.

\begin{lemma}[Katz-Tao Lemma]
Let $A$ be a finite subset of some field $k$ and suppose that $|A \cdot A|, |A + A| \leq K|A|$. Then there is a set $A' \subseteq A$ with $|A'| \geq K^{-C}|A|$ such that for any rational function $\psi : k^m \rightarrow k \cup \{\infty\}$ we have the estimate $|\psi(A')| \leq K^{O_{\psi}(1)}|A'|$, where
\[ \psi(A') := \{ \psi(a_1,\dots, a_m) : a_1,\dots, a_m \in A'\}.\]
\end{lemma}

In proving Theorem \ref{sum-product-fp} we may replace $A$ by this set $A'$. Henceforth, then, we assume that
\begin{equation}\label{rat-funct-bd} |\psi(A)| \leq K^{O_{\psi}(1)}|A|\end{equation} for all rational functions $\psi : \F_p^m \rightarrow \F_p \cup \{\infty\}$. We may also assume that $0 \notin A$.

Consider the set $X \subseteq \SL_2(\F_p)$ defined by
\[ X := \left\{ \begin{pmatrix} a_1 & a_2 \\ a_3 & \frac{1 + a_2 a_3}{a_1}\end{pmatrix} : a_1, a_2, a_3 \in A.\right\}\]
If one imagines $A$ to be an ``approximate subfield'' of $\F_p$ then $X$ might be considered thought of as an appropriate definition of $\SL_2(A)$.

It is clear from \eqref{rat-funct-bd} that $|X^3| \leq K^C|X|$. Applying Corollary \ref{helfgott-type-expansion} with $G = \SL_2(\F_p)$ it follows that one of the following occurs:
\begin{enumerate}
\item  $|X| \leq K^C$;
\item  $|X| \geq K^{-C}|\SL_2(\F_p)|$, or
\item $X$ does not generate $\SL_2(\F_p)$.
\end{enumerate}
Obviously (i) implies that $|A| \leq K^{C'}$, whilst (ii) implies that $|A| \geq K^{-C'}p$. It remains, then, to rule out (iii). For this we use Dickson's classification of subgroups of $\SL_2(\F_p)$ (see \cite{dickson}). In fact, we only need the following consequence, reported for instance in \cite[Proposition 3]{bourgain-gamburd}.

\begin{proposition}
Suppose that $H$ is a proper subgroup of $\SL_2(\F_p)$ with $|H| > 60$. Then $H$ is 2-step solvable, and hence
\begin{equation}\label{comm} [[h_1, h_2], [h_3,h_4]] = \id\end{equation} for all $h_1,h_2,h_3,h_4 \in H$.
\end{proposition}

Supposing that $X$ generates such a group $H$, consider the condition \eqref{comm} with $h_1,h_2,h_3,h_4$ being arbitrary elements of $X$ parametrised by twelve elements $a_1,\dots, a_{12}$ of $A$. Clearing denominators, this yields a polynomial $\phi : \F_p^{12} \rightarrow \F_p$ of degree $O(1)$ which vanishes identically on $A^{12}$. Since $\SL_2(\F_p)$ is not solvable, this polynomial cannot vanish identically. We may then apply the following well-known lemma, which is proved in \cite[Lemma 2.1]{combinatorial-nullstellensatz}.

\begin{lemma}
Let $\F$ be any field and suppose that $\phi : \F^m \rightarrow \F$ is a polynomial in variables $x_1,\dots,x_m$. Suppose that the degree $\deg_i \phi$ in each of the variables $x_i$ is at most $d$, and that there is a set $S \subseteq \F$ with $|S| > d$ such that $\phi(s_1,\dots, s_m) = 0$ whenever $s_1,\dots, s_m \in S$. Then $\phi$ is identically zero.
\end{lemma}

It follows immediately from this that condition (iii) can only hold if $|A| = O(1)$. This completes the proof of Theorem \ref{sum-product-fp}.\hfill $\Box$

\appendix

\section{Quantitative algebraic geometry via ultrafilters}\label{compact-sec}

The purpose of this appendix is to use ultrafilter methods to establish the quantitative algebraic geometry results in Lemmas \ref{zar}, \ref{compos},\ref{bezout}, \ref{dimlem}, \ref{slice-lem}, \ref{centralem}, \ref{zark} and Lemma \ref{escape}.  In principle, all of these arguments could be replaced by more involved effective analogues, but we do not do this here. Doing so would make all the constants used in the main theorem (Theorem \ref{mainthm2} computable in principle.

The key point here is that many basic concepts in the theory of algebraic varieties (such as dimension, irreducibility, or degree) are continuous with respect to the operation of taking ultralimits\footnote{This is analogous, but not quite identical to, the more well-known fact that such concepts are continuous with respect to operations such as projective limits; see \cite{ega}.}.  As a consequence of this and a compactness-type argument, many qualitative statements regarding such concepts automatically (but ineffectively) have quantitative analogues that are uniform over all choice of (algebraically closed) base field and all choices of coefficients used to define the varieties at hand.

We turn to the details.  We will need a \emph{non-principal ultrafilter} $\alpha_\infty \in \beta \N \backslash \N$, i.e. a collection of subsets of $\N$ with the following properties:
\begin{itemize}
\item No finite set lies in $\alpha_\infty$.
\item If $A \subset \N$ is in $\alpha_\infty$, then any subset of $\N$ containing $A$ is in $\alpha_\infty$.
\item If $A, B$ lie in $\alpha_\infty$, then $A \cap B$ also lies in $\alpha_\infty$.
\item If $A \subset \N$, then exactly one of $A$ and $\N \backslash A$ lies in $\alpha_\infty$.
\end{itemize}
An easy consequence of these axioms is that if $A \cup B \in \alpha_{\infty}$ then at least one of $A$ and $B$ lies in $\alpha_{\infty}$. Given a property $P(\alpha)$ which may be true or false for each natural number $\alpha$, we say that $P$ is true for $\alpha$ \emph{sufficiently close to} $\alpha_\infty$ if the set $\{ \alpha \in \N: P(\alpha) \hbox{ holds}\}$ lies in $\alpha_\infty$.  The existence of a non-principal ultrafilter $\alpha_\infty$ is guaranteed by the axiom of choice.  We fix $\alpha_\infty$ throughout the rest of this appendix.  We make the important observation that if $x_\alpha$ takes on only a finite number of values for $\alpha$ sufficiently close to $\alpha_\infty$ (e.g. if $x_\alpha$ ranges in a discrete space and is also bounded), then it is in fact \emph{constant} for $\alpha$ sufficiently close to $\alpha_\infty$.  We write this constant as $\lim_{\alpha \to \alpha_\infty} x_\alpha$.

Given any sequence $(X_\alpha)_{\alpha \in \N}$ of sets, we define the \emph{ultraproduct} $\prod_{\alpha \to \alpha_\infty} X_\alpha$ to be the space of equivalence classes of tuples $(x_\alpha)_{\alpha \in \N}$ with $x_\alpha \in X_\alpha$ for all $\alpha \in \N$, where the equivalence relation is given by requiring $(x_\alpha)_{\alpha \in \N}$ and $(y_\alpha)_{\alpha \in \N}$ to be equivalent if $x_\alpha =y_\alpha$ for all $\alpha$ sufficiently close to $\alpha_\infty$.  We call such an equivalence class the \emph{ultralimit} of the sequence $(x_\alpha)_{\alpha \in \N}$ and write it as $\lim_{\alpha \to \alpha_\infty} x_\alpha$.  Note that the ultraproduct $\prod_{\alpha \to \alpha_\infty} X_\alpha$ and ultralimit $\lim_{\alpha \to \alpha_\infty} x_\alpha$ remain well defined even if $X_\alpha$ or $x_\alpha$ are only defined for $\alpha$ sufficiently close to $\alpha_\infty$.

We make the basic observation that two ultraproducts \[ \prod_{\alpha \to\alpha_\infty} X_\alpha,\prod_{\alpha \to\alpha_\infty} Y_\alpha\] agree\footnote{Strictly speaking, this statement is not precisely correct, because the equivalence relations used to define the two ultraproducts are defined on distinct domains.  But if one extends the equivalence relation to a common domain, such as $(\bigcup_\alpha X_\alpha \cup Y_\alpha)^\N$, then the statement becomes valid.  Alternatively, one can embed all the groups, fields, varieties, etc. one is interested in studying in a single \emph{standard universe} ${\mathcal U}$, which is assumed to be a set, and define the ultralimit equivalence relation on all sequences $(x_\alpha)_{\alpha \in \alpha_\infty}$ in ${\mathcal U}^\N$ (or more generally, on sequences in ${\mathcal U}$ defined for $\alpha$ sufficiently close to $\alpha_\infty$).  We will not dwell on this foundational issue in the rest of this paper, as it makes no significant impact on the actual arguments here.}  if and only if $X_\alpha=Y_\alpha$ for all $\alpha$ sufficiently close to $\alpha_0$.

Any operation or relation on a sets $X_\alpha$, $\alpha \in \N$ carries over to the ultraproduct, and a famous theorem of {\L}os asserts that any statement in first-order logic that is true for all the $X_\alpha$ (or for $X_\alpha$ with $\alpha$ sufficiently close to $\alpha_\infty$) is also true in the ultralimit.  For instance, if $k_\alpha$ is a sequence of algebraically closed fields, then the ultraproduct $k := \prod_{\alpha \to \alpha_\infty} k_\alpha$ is also an algebraically closed field, because the property of being an algebraically closed field can be expressed as a set of first-order sentences involving the field operations.

Note that if $k := \prod_{\alpha \to \alpha_\infty} k_\alpha$ is an ultraproduct of algebraically closed fields, then we have some canonical identifications
$$ \A^n(k) \equiv \prod_{\alpha \to \alpha_\infty} \A^n(k_\alpha)$$
and
$$ \P^n(k) \equiv \prod_{\alpha \to \alpha_\infty} \P^n(k_\alpha)$$
for fixed $n$.  Furthermore, if $V_\alpha$ is a sequence of affine varieties in $\A^n(k_\alpha)$, with complexity bounded uniformly in $\alpha$, we see that the ultraproduct $V := \prod_{\alpha \to \alpha_\infty} V$ is an affine variety in $\A^n(k)$. This is basically because the ultralimit of polynomials of bounded degree remains polynomial.  Conversely, every affine variety in $\A^n(k)$ can be expressed as an ultralimit of affine varieties in $\A^n(k_\alpha)$ of uniformly bounded complexity for $\alpha$ sufficiently close to $\alpha_\infty$.  Similar claims of course hold for projective varieties, quasiprojective varieties, and constructible sets.

Now we investigate various continuity properties of the ultralimit of varieties.  We begin with the continuity of dimension.

\begin{lemma}[Continuity of dimension]\label{dimcont} Let $k_\alpha$ be a sequence of algebraic- ally closed fields, let $n \geq 1$, and let $V_\alpha \subset \A^n(k_\alpha)$ be a family of affine algebraic varieties of complexity uniformly bounded for $\alpha$ sufficiently close to $\alpha_\infty$.  Write $k := \prod_{\alpha \to \alpha_\infty} k_\alpha$ and $V := \prod_{\alpha \to\alpha_\infty} V_\alpha$.  Then $\dim(V) = \lim_{\alpha \to \alpha_\infty} \dim(V_\alpha)$.  In other words, we have $\dim(V) = \dim(V_\alpha)$ for all $\alpha$ sufficiently close to $\alpha_\infty$.

Similarly with affine varieties replaced by projective varieties, quasiprojective varieties, or constructible sets.
\end{lemma}

\begin{proof}  We begin with the claim for affine varieties.

We induct on dimension $n$.  The case $n=0$ is trivial, so suppose that $n \geq 1$ and the claim has already been shown for $n-1$.  Write $d$ for the dimension of $V$.  If $d=-1$, then $V$ is empty and so $V_\alpha$ must be empty for all $\alpha$ sufficiently close to $\alpha_\infty$, so suppose that $d \geq 0$.  Since $V$ has dimension $d$, we see from standard algebraic geometry theory that the slices
$$ V_t := \{ x \in \A^{n-1}(k_\alpha): (x,t) \in V \}$$
all have dimension $d-1$ (or are all empty) for all but finitely many values $t_1,\ldots,t_r$ of $t \in k$, and the exceptional slices $V_{t_i}$ have dimension at most $d$.  If the $V_t$ for $t \neq t_1,\ldots,t_r$ are all empty, then one of the exceptional slices $V_{t_i}$ has to have dimension exactly $d$.  As $k$ is the ultraproduct of the $k_\alpha$, we can write $t_i = \lim_{\alpha \to \alpha_\infty} t_{\alpha,i}$ for each $1 \leq i \leq r$.

Suppose first that the $V_t$ have dimension $d-1$ for all $t \neq t_1,\ldots,t_r$. We claim that for $\alpha$ sufficiently close to $\alpha_\infty$, the slices $(V_\alpha)_{t_{\alpha}}$ have dimension $d-1$ whenever $t_{\alpha} \neq t_{\alpha,1},\ldots,t_{\alpha,r}$.  Indeed, suppose that this were not the case.  Carefully negating the quantifiers (and using the ultrafilter property), we see that for $\alpha$ sufficiently close to $\alpha_\infty$, we can find $t_{\alpha} \neq t_{\alpha,1},\ldots,t_{\alpha,r}$ such that $(V_\alpha)_{t_{\alpha}}$ has dimension different from $d-1$.  Taking ultraproducts and writing $t := \lim_{\alpha \to \alpha_\infty} t_\alpha$, we see from the induction hypothesis that $V_t$ has dimension different from $d-1$, contradiction.  Thus $(V_\alpha)_t$ has dimension $d-1$ whenever $t \neq t_{\alpha,1},\ldots,t_{\alpha,r}$ and $\alpha$ is sufficiently close to $\alpha_\infty$ (uniformly in $t$).  A similar argument also shows that $(V_\alpha)_t$ has dimension at most $d$ whenever $t \in \{t_{\alpha,1},\ldots,t_{\alpha,r}\}$  and $\alpha$ is sufficiently close to $\alpha_\infty$.  Hence $V_\alpha$ has dimension exactly $d$.

A similar argument applies when $V_t$ is empty for all $t \neq t_1,\ldots,t_r$, and has dimension exactly equal to $d$ for at least one of the $t_1,\ldots,t_r$, and at most $d$ for the other slices.  This establishes the claim for affine varieties.

The projective case follows from the affine one by covering projective space by finitely many copies of affine space.

The quasiprojective case follows from the projective one by expressing a quasiprojective variety $V$ of dimension $d$ as the projective variety of dimension at most $d$, with a projective variety of dimension at most $d-1$ removed; this can be achieved by starting with the Zariski closure of $V$ and decomposing into irreducible components.

The constructible set case then follows from the quasiprojective one by expressing a constructible set of dimension $d$ as the union of finitely many quasiprojective varieties of dimension at most $d$, with at least one of these varieties having dimension exactly $d$.
\end{proof}

This already gives a quick proof of Lemma \ref{bezout}:

\emph{Proof of Lemma \ref{bezout}.} Suppose this lemma failed.  Carefully negating the quantifiers (and using the axiom of choice), we may find a dimension $n$, a sequence $V_\alpha \subset \A^n(k_\alpha)$ of dimension $0$ constructible sets and uniformly bounded complexity over algebraically closed fields $k_\alpha$, such that $|V_\alpha| \to \infty$ as $\alpha \to \infty$.  We pass to an ultralimit to obtain a constructible set $V := \prod_{\alpha \to \alpha_\infty} V_\alpha$, which by Lemma \ref{dimcont} has dimension $0$, and is thus finite.  But then this forces $V_\alpha$ to be finite for $\alpha$ sufficiently close to $\alpha_\infty$ (indeed we have $|V_\alpha| = |V|$ in such a neighbourhood), contradiction.\hfill $\Box$\vspace{11pt}

Now we study continuity of irreducibility.  We will shortly establish the following result.

\begin{lemma}[Continuity of irreducibility]\label{cont-irred}  Suppose that $V_\alpha \subset \A^n(k_\alpha)$ are affine varieties of uniformly bounded complexity over algebraically closed fields $k_\alpha$, and let $V := \prod_{\alpha \to \alpha_\infty} V_\alpha$ be the ultraproduct.  Then $V$ is irreducible if and only if $V_\alpha$ is irreducible for all $\alpha$ sufficiently close to $\alpha_\infty$.  Similarly for projective or quasiprojective varieties instead of affine varieties.
\end{lemma}
\newcommand\Grass{\operatorname{Grass}}
This result is however a little bit more difficult to establish than Lemma \ref{dimcont}, because it requires one to understand the relationship between the complexity of a variety and its degree.  Recall that the \emph{degree} of an affine variety $V \subset \A^n(k)$ of dimension $d$ is the cardinality of $|V \cap W|$, where $W$ is a generic affine $n-d$-dimensional subspace of $\A^n(k)$ (i.e. for all $W$ in the affine Grassmanian $\Grass_{n,n-d}(k)$ of affine $n-d$-dimensional affine subspaces of $\A^n(k)$, outside of a subvariety of
$\Grass_{n,n-d}(k)$ of strictly smaller dimension).  It is a standard fact in algebraic geometry that the degree is well-defined as a natural number.  We also have the following nontrivial fact.

\begin{theorem}[Degree controls complexity]\label{klei}  Let $V$ be an \emph{irreducible} affine variety in $\A^n(k)$ of degree $D$  over an algebraically closed field $k$.  Then $V$ has complexity at most $C_{n,D}$ for some constants $n, D$ depending only on $n, D$.
\end{theorem}

\begin{proof}\footnote{We thank Jordan Ellenberg and Ania Otwinowska for this argument, which goes back to  \cite{mumford}.}  It suffices to show that $V$ can be cut out by polynomials of degree $D$, since the space of polynomials of degree $D$ that vanish on $V$ is a vector space of dimension bounded only by $n$ and $D$.

Let $V$ have dimension $d$.  We pick a generic affine subspace $W$ of $k^n$ of dimension $n-d-2$, and consider the cone $C(W,V)$ formed by taking the union of all the lines joining a point in $W$ to a point in $V$.  This is an algebraic image of $W \times V \times k$ and is thus generically an algebraic set of dimension $n-1$, i.e. a hypersurface.  Furthermore, as $V$ has degree $D$, it is not hard to see that $C(W,V)$ has degree $D$ as well.  Since a hypersurface is necessarily cut out by a single polynomial, this polynomial must have degree $D$.

To finish the claim, it suffices to show that the intersection of the $C(W,V)$ as $W$ varies is exactly $V$.  Clearly, this intersection contains $V$.  Now let $p$ be any point not in $V$.  The cone of $V$ over $p$ can be viewed as an algebraic subset of the projective space $\P^{n-1}(k)$ of dimension $d$; meanwhile, the cone of a generic subspace $W$ of dimension $n-d-2$ is a generic subspace of $\P^{n-1}(k)$ of the same dimension.  Thus, for generic $W$, these two cones do not intersect, and thus $p$ lies outside $C(W,V)$, and the claim follows.
\end{proof}

\emph{Remark.} There is a stronger and more difficult theorem that asserts that if the degree of a \emph{scheme} in $k^n$ is bounded, then the complexity of that scheme is bounded as well; see \cite[Corollary 6.11]{klei:sga6}.  We will not need this stronger statement here.  The converse statement (that complexity controls degree) is also true, being a corollary of Lemma \ref{bezout}.

\emph{Proof of Lemma \ref{cont-irred}.} We first establish this claim for affine varieties.
The ``if'' direction of the lemma is the easiest. Suppose then that $V$ is reducible. Then it is the proper union of affine varieties $V_1, V_2$.  Each $V_1, V_2$ can be expressed as the ultralimit of affine varieties $V_{\alpha,1}$, $V_{\alpha,2}$ of bounded complexity, and one easily sees that $V_\alpha$ is the proper union of $V_{\alpha,1}$ and $V_{\alpha,2}$ for $\alpha$ sufficiently close to $\alpha_\infty$. Thus $V_\alpha$ is reducible for such $\alpha$, and the claim follows.

Now suppose that $V$ is irreducible; our task is to show that the $V_\alpha$ are irreducible for $\alpha$ sufficiently close to $\alpha_0$.

Let $d$ and $D$ be the dimension and degree of $V$, thus $|V \cap W| = D$ for generic $W \in \Grass(k)$.  Undoing the ultralimit using Lemma \ref{dimcont}, we see that for $\alpha$ sufficiently close to $\alpha_0$, $|V_\alpha \cap W_\alpha| = D$ for generic $W_\alpha \in \Grass(k_\alpha)$.  In other words, $V_\alpha$ has degree $D$. The same is clearly true of any irreducible subvariety of $U_{\alpha} \subseteq V_{\alpha}$ with $\dim(U_{\alpha}) = \dim(V_{\alpha})$. By Theorem \ref{klei}, then, any such subvariety $U_{\alpha}$ will have complexity bounded by $C_{n,D}$ uniformly in $\alpha$.  For each $\alpha$ select some $U_{\alpha}$, and let $U$ be the ultraproduct these $U_\alpha$.  Then by Lemma \ref{dimcont} and the uniform complexity bound, $U$ is a $d$-dimensional subvariety of $V$, and thus must equal all of $V$ by the irreducibility of $V$.  But this implies that $U_\alpha=V_\alpha$ for all $\alpha$ sufficiently close to $\alpha_0$, and the claim for affine varieties then follows.

The claim for projective varieties then follows by covering projective space by a finite number of copies of affine space.  The claim for quasiprojective varieties then follows by writing an irreducible quasiprojective variety as an irreducible projective variety (i.e. the Zariski closure) with some varieties of strictly smaller dimension removed.\hfill $\Box$\vspace{11pt}

This has the following consequence, used at several points in the paper.

\begin{lemma}\label{quivar}  Let $V \subset \A^n(k)$ be an affine variety of complexity at most $M$ over an algebraically closed field $k$.  Then $V$ can be expressed as the union of at most $O_M(1)$ irreducible varieties of complexity at most $O_M(1)$.

Similarly with affine varieties replaced by projective or quasiprojective varieties.
\end{lemma}

\begin{proof} As $n$ is bounded by $M$, it suffices to prove the claim for a fixed $n$.  We shall just establish the claim for affine varieties, as the projective and quasiprojective cases are similar.

Fix $n$ and $M$, and suppose the claim failed.  Carefully negating all the quantifiers (and using the axiom of choice), we see that there exists a sequence $V_\alpha\subset \A^n(k_\alpha)$ of affine varieties of uniformly bounded complexity, such that $V_\alpha$ cannot be expressed as the union of $\alpha$ or fewer irreducible affine varieties of complexity at most $\alpha$.  Now we pass to an ultralimit, obtaining an affine variety $V := \lim_{\alpha \to \alpha_\infty} V_\alpha \subset \A^n(k)$.  As $V$ is an affine variety, standard algebraic geometry allows one to write $V$ as the union of finitely many irreducible affine varieties $V_1,\ldots,V_m$.  Each of these varieties $V_i$ is the ultraproduct of affine varieties $V_{\alpha,i} \subset \A^n(k_\alpha)$ of bounded complexity; by Lemma \ref{cont-irred}, the $V_{\alpha,i}$ will be irreducible for $\alpha$ sufficiently close to $\alpha_\infty$.  On the other hand, the $V_\alpha$ are the union of the $V_{\alpha,i}$ for $\alpha$ sufficiently close to $\alpha_\infty$.  This contradicts the construction of the $V_\alpha$, and the claim follows.
\end{proof}

As a consequence of the above machinery we have the following result asserting the continuity of Zariski closure and Zariski density.

\begin{lemma}\label{cont-zarisk}  Suppose that $V_\alpha \subset \P^n(k_\alpha)$ are varieties of uniformly bounded complexity over algebraically closed fields $k_\alpha$, and let $V := \prod_{\alpha \to \alpha_\infty} V_\alpha$ be the ultraproduct.  Then $\overline{V} = \prod_{\alpha \to \alpha_\infty} \overline{V_\alpha}$, where $\overline{V}$ and $\overline{V_\alpha}$ are the Zariski closures of $V, V_\alpha$ respectively.

Similarly, if $V_\alpha \subset W_\alpha$ are varieties of uniformly bounded complexity over $k_\alpha$, and $V := \prod_{\alpha \to \alpha_\infty} V_\alpha$ and $W := \prod_{\alpha \to \alpha_\infty} W_\alpha$, then $V$ is Zariski-dense in $W$ if and only if there exists a finite $M$ such that $V_\alpha$ is $M$-Zariski-dense in $W_\alpha$ for all $\alpha$ sufficiently close to $\alpha_\infty$.
\end{lemma}

\begin{proof}  We just prove the first claim, as the second claim is similar.
One can represent $V$ as the union of Zariski-dense open subsets $V_i$ of irreducible projective varieties $\overline{V_i}$, with $\overline{V}$ then being the union of the $\overline{V_i}$.  One can then view $V_i$ as the irreducible projective variety $\overline{V_i}$ with a subvariety of strictly smaller dimension removed.  We can express the variety $\overline{V_i}$ as an ultraproduct $\prod_{\alpha \to \alpha_\infty} \overline{V_i}_\alpha$.
Using Lemma \ref{dimcont} and Lemma \ref{cont-irred}, we thus see that for $\alpha$ sufficiently close to $\alpha_\infty$, the $\overline{V_i}_\alpha$ are irreducible, and $V_\alpha$ is the union of Zariski-dense subsets of $\overline{V_i}_\alpha$.  Thus $\overline{V_\alpha} = \bigcup_i \overline{V_i}_\alpha$ for such $\alpha$, and the claim follows.
\end{proof}

Now we can quickly prove Lemma \ref{zar}.

\emph{Proof of Lemma \ref{zar}.} Suppose this lemma failed.  Carefully negating the quantifiers (and using the axiom of choice), we may find a dimension $n$, a sequence $V_\alpha \subset \P^n(k_\alpha)$ of constructible sets and uniformly bounded complexity over algebraically closed fields $k_\alpha$, such that the Zariski closure $\overline{V_\alpha}$ of $V_\alpha$ has complexity at least $\alpha$.  We pass to an ultralimit to obtain a constructible set $V := \prod_{\alpha \to \alpha_\infty} V_\alpha$.  By Lemma \ref{cont-zarisk}, the Zariski closure $\overline{V}$ is the ultraproduct of the $\overline{V_\alpha}$, and hence the $\overline{V_\alpha}$ have bounded complexity, a contradiction.\hfill $\Box$\vspace{11pt}

Now we apply the ultralimit machinery to regular maps.  Given a collection of maps $\phi_\alpha: X_\alpha \to Y_\alpha$ for $\alpha$ sufficiently close to $\alpha_\infty$, we can construct the ultralimit $\phi := \lim_{\alpha \to \alpha_\infty} \phi_\alpha$, defined as the map $\phi: X \to Y$ from the ultraproduct $X := \prod_{\alpha \to \alpha_\infty} X_\alpha$ to the ultraproduct $Y := \prod_{\alpha \to \alpha_\infty} Y_\alpha$ by the formula
$$ \phi( \lim_{\alpha \to \alpha_\infty} x_\alpha ) := \lim_{\alpha \to \alpha_\infty} \phi_\alpha(x_\alpha)$$
for any sequence $x_\alpha \in X_\alpha$.  It is easy to see that this limit map is well defined.

From Definition \ref{regmap} we obtain the following basic lemma.

\begin{lemma}[Ultralimits of regular maps]\label{ultralim}  For $\alpha$ sufficiently close to $\alpha_\infty$, let $V_\alpha, W_\alpha$ be varieties of uniformly bounded complexity over an algebraically closed field $k_\alpha$, and let $\phi_\alpha: V_\alpha \to W_\alpha$ be a regular map of uniformly bounded complexity.  Then the ultralimit $\phi := \lim_{\alpha \to \alpha_\infty} \phi_\alpha$ is a regular map from the ultraproduct $V := \prod_{\alpha \to \alpha_\infty} V_\alpha$ \textup{(}which is a variety over $k := \prod_{\alpha \to \alpha_\infty} k_\alpha$\textup{)} to the ultraproduct $W := \prod_{\alpha \to \alpha_\infty} W_\alpha$ \textup{(}which is also a variety over $k$\textup{)}.

Conversely, if $V := \prod_{\alpha \to \alpha_\infty} V_\alpha$ and $W := \prod_{\alpha \to \alpha_\infty} W_\alpha$ are algebraic varieties over $k = \prod_{\alpha \to \alpha_\infty} k_\alpha$, and $\phi: V \to W$ is a regular map, then one can write $\phi = \lim_{\alpha \to \alpha_\infty} \phi_\alpha$, where for $\alpha$ sufficiently close to $\alpha_\infty$, $\phi_\alpha:  V_\alpha \to W_\alpha$ is a regular map of complexity bounded uniformly in $\alpha$.
\end{lemma}

Now we can prove Lemma \ref{compos}.

\emph{Proof of Lemma \ref{compos}.} We begin with the first claim.  If this claim failed, then there exists a two sequences of regular maps $\phi_\alpha: V_\alpha \to W_\alpha$ and $\psi_\alpha: U_\alpha \to V_\alpha$ of uniformly bounded complexity (so in particular $U_\alpha,V_\alpha,W_\alpha$ also have uniformly bounded complexity), such that $\psi_\alpha \circ \phi_\alpha: U_\alpha \to W_\alpha$ is not given by a regular map of complexity at most $\alpha$.  But then by Lemma \ref{ultralim} we may take ultralimits and ultraproducts to create two regular maps $\phi: V \to W$ and $\psi: U \to V$ in the obvious manner.  From classical algebraic geometry we know that the composition $\psi \circ \phi: U \to W$ is then regular, and so by another application of Lemma \ref{ultralim} we see that $\psi_\alpha \circ \phi_\alpha: U_\alpha \to W_\alpha$ is a regular map of bounded complexity for $\alpha$ sufficiently close to $\alpha_\infty$, giving the desired contradiction.

Now we prove the second claim, for images.  If this claim failed, then we can find a sequence of regular maps $\phi_\alpha: V_\alpha \to W_\alpha$ of uniformly bounded complexity, such that $\phi_\alpha(V_\alpha)$ is not a constructible set of complexity at most $\alpha$.  Taking ultralimits again, we obtain a regular map $\phi: V \to W$.  From classical algebraic geometry we know that $\phi(V)$ is a constructible set, which implies that for $\alpha$ sufficiently close to $\alpha_0$, $\phi_\alpha(V_\alpha)$ is a constructible set of uniformly bounded complexity, giving the desired contradiction.  The analogous claim for pre-images is proven similarly.\hfill $\Box$\vspace{11pt}

Now we can establish continuity of dominance under ultralimits.

\begin{lemma}[Continuity of dominance]\label{cont-dom}  Let $\phi_\alpha: V_\alpha \to W_\alpha$ be a sequence of regular maps of uniformly bounded complexity over an algebraically closed field $k_\alpha$.  Let $\phi :=\lim_{\alpha \to \alpha_\infty} \phi_\alpha$, $V := \prod_{\alpha \to \alpha_\infty} V_\alpha$, $W := \prod_{\alpha \to \alpha_\infty} W_\alpha$, and $k := \prod_{\alpha \to \alpha_\infty} k_\alpha$.  Then $\phi: V \to W$ is dominant if and only if $\phi_\alpha: V_\alpha \to W_\alpha$ is dominant for all $\alpha$ sufficiently close to $\alpha_\infty$.
\end{lemma}

\begin{proof}  Suppose first that $\phi_\alpha$ is dominant for all $\alpha$ sufficiently close to $\alpha_\infty$.  Then for $\alpha$ sufficiently close to $\alpha_\infty$, $V_\alpha$ is irreducible, hence $V$ is irreducible by Lemma \ref{cont-irred}.  The Zariski closure of $\phi_\alpha(V_\alpha)$ is $W_\alpha$, and hence by Lemma \ref{cont-zarisk} the Zariski closure of $\phi(V)$ is $W$. Therefore $\phi$ is dominant.

The converse claim follows by reversing all of these steps.
\end{proof}

Now we prove Lemma \ref{dimlem}.  The key is to establish the following qualitative variant.

\begin{lemma}[Qualitative dimension lemma]\label{dimlem-quali}   Let $V, W$ be varieties, and let $\phi: V \to W$ be a regular map.  Then there exists a Zariski open subset $V'$ of $V$, and a subvariety $W'$ of $W$ of dimension at most $\dim(V)$, with the following two properties:
\begin{itemize}
\item \textup{(}Generic mapping\textup{)} $\phi(V') \subset W'$.
\item \textup{(}Generic fibres\textup{)}  For any $w \in W'$, the set $\{ v \in V': \phi(v)=w \}$ is a constructible set of dimension at most $\dim(V) - \dim(W')$.
\end{itemize}
If $\phi$ is dominant, then we may take $W'$ to be a Zariski-dense subset of $W$.
\end{lemma}

\begin{proof}  Let $d$ be the dimension of $V$.  We may view $V$ as the union of Zariski-open subsets of irreducible projective varieties $V_i$ of dimension at most $d$.  We may remove all varieties $V_i$ of dimension strictly less than $d$ by placing them in the exceptional set $V \backslash V'$.  We may then work with just a single $V_i$, as the general case follows by taking unions.  Thus, $V$ is now a Zariski-dense subvariety of an irreducible projective variety $\overline{V}$.

We may replace $W$ by the Zariski closure of $\phi(V)$.  When one does so, $\phi$ becomes a dominant map; also we now have $\dim(W) \leq \dim(V)$.  But now, a standard result in classical algebraic geometry (see \cite[\S I.6.3]{shafarevich}) shows that there exist Zariski-open subsets $V', W'$ of $V, W$ respectively, such that $\phi$ restricts to a
dominant map $\phi': V' \to W'$, whose fibres $\{ v \in V': \phi(v)=w\}$ have dimension $\dim(V)-\dim(W)$.  The claims then follow.
\end{proof}

\emph{Proof of Lemma \ref{dimlem}.} This is deduced from Lemma \ref{dimlem-quali} by a very similar ultralimit argument to previous arguments, so we shall only give a sketch here.  If the first claim failed, then one can find a sequence of regular maps $\phi_\alpha: V_\alpha \to W_\alpha$ of uniformly bounded complexity, such that one cannot find a $\alpha$-Zariski open subset $V'_\alpha$ of $V_\alpha$ and a subvariety $W'_\alpha$ of dimension at most $\dim(V)$ with complexity at most $\alpha$ such that $\phi(V'_\alpha) \subset W'_\alpha$ and for every $w \in W'_\alpha$, the set  $\{ v \in V'_\alpha: \phi(v)=w \}$ is a constructible set of dimension at most $\dim(V_\alpha) - \dim(W'_\alpha)$ and complexity at most $\alpha$.  One then takes ultralimits of the $\phi_\alpha, V_\alpha, W_\alpha$ to create a regular map $\phi: V \to W$.  Applying Lemma \ref{dimlem-quali} and undoing the ultralimit (using Lemma \ref{dimcont}) we obtain the required contradiction.

The argument when the second claim fails is similar, except that now the $\phi_\alpha$ are also dominant maps, and $W'_\alpha$ is required to be $\alpha$-Zariski dense in $W$.  One then argues as before but also uses Lemma \ref{cont-dom} and Lemma \ref{cont-zarisk}.\hfill $\Box$\vspace{11pt}

In a similar spirit, Lemma \ref{slice-lem} follows by the usual ultralimit argument from the following qualitative lemma.

\begin{lemma}[Qualitative slicing lemma]\label{slice-lem-quali}   Let $V, W$ be varieties, and let $S$ be a subvariety of $V \times W$ of dimension strictly less than $\dim(V) + \dim(W)$.  Then for generic $v \in V$, the set $\{ w \in W: (v,w) \in S \}$ is a constructible set of dimension strictly less than $\dim(W)$.
\end{lemma}

\begin{proof}  The projection map $\pi: S \to V$ that maps $(v,w)$ to $v$ is a regular map.  Let $V'$ be the effective image of this map, given by Lemma \ref{dimlem-quali}.  The claim then follows from that lemma (dividing into two cases, depending on whether the effective image $V'$ has dimension equal to that of $V$, or has strictly lower dimension).
\end{proof}

Finally, we discuss algebraic groups.  From Definition \ref{algdef} and Lemma \ref{ultralim} we have the following lemma.

\begin{lemma}[Ultralimits of algebraic groups]\label{ultralim-gp}  For $\alpha$ sufficiently close to $\alpha_\infty$, let $G_\alpha$ be an algebraic group of uniformly bounded complexity over an algebraically closed field $k_\alpha$.  Then the ultraproduct $G := \prod_{\alpha \to \alpha_\infty} G_\alpha$ is an algebraic group also.  Conversely, every algebraic group over an ultraproduct $G = \prod_{\alpha \to \alpha_\infty} k_\alpha$ of algebraically closed fields is an ultraproduct of algebraic groups $G_\alpha$ of bounded complexity over $k_\alpha$.
\end{lemma}

It is clear that if $G$ is an algebraic group, then the centralisers $Z(a)$ and conjugacy classes $a^G$ are constructible sets, and that the normaliser and centraliser of an algebraic group $H$ is another algebraic group; and so Lemma \ref{centralem} follows from Lemma \ref{ultralim-gp} and the usual ultralimit argument.

Now we turn to Lemma \ref{zark}.  This will be deduced from the following qualitative fact.

\begin{lemma}[Zariski closure of groups]\label{zclg}  Let $A \subset G$ be any subgroup of an algebraic group $G$.  Then the Zariski closure $\overline{A}$ of $A$ is an algebraic subgroup of $G$.
\end{lemma}

\begin{proof} $A$ contains the identity and is closed with respect to inverses, so $\overline{A}$ is also.  For any $a \in A$, the map $g \mapsto ag$ is a regular isomorphism that preserves $A$, and thus must also preserve $\overline{A}$; thus $a \overline{A} = \overline{A}$.  On the other hand, the set $\{ g \in G: g \overline{A} = \overline{A} \}$ is a variety (being the intersection of a family of closed subvarieties in $G$) that contains $A$, and so must contain $\overline{A}$.  Thus $\overline{A}$ is closed under multiplication, and the claim follows.
\end{proof}

\emph{Proof of Lemma \ref{zark}.} Suppose the claim failed.  Then there exists an $M$ and a sequence of algebraic groups $G_\alpha$ of complexity at most $M$, subvarieties $V_\alpha$ of complexity at most $M$, and symmetric subsets $A_\alpha$ containing $\id$ in $G_\alpha$ with $A^\alpha_\alpha \subset V_\alpha$, such that $A_\alpha$ is not contained in any algebraic subgroup of $G_\alpha$ contained in $V_\alpha$ of complexity at most $\alpha$.   Let $G, A, V$ be the ultraproducts of $G_\alpha, A_\alpha, V_\alpha$ respectively.  By construction, $A$ contains the identity, is closed with respect to inverses, and $A^m \subset V$ for any finite $m$.  In particular, the group $\langle A \rangle$ generated by $A$ is contained in $V$.  By Lemma \ref{zclg}, the Zariski closure $H := \overline{\langle A \rangle}$ of this group is an algebraic subgroup of $G$ contained in $V$.  Undoing the ultraproduct using Lemma \ref{ultralim-gp}, we can express $H$ as the ultraproduct of algebraic subgroups $H_\alpha$ of $G_\alpha$ which contain $A$ and are of bounded complexity for $\alpha$ sufficiently close to $\alpha_\infty$.  But this contradicts the construction of $A_\alpha$.\hfill $\Box$\vspace{11pt}

By using very similar ultralimit arguments to those already employed, we see that Lemma \ref{escape} follows from the next result, which might be called a qualitative product-conjugation phenomenon for subvarieties.

\begin{lemma}\label{escape-2}  Let $G$ be an algebraic group.  Let $V, W$ be algebraic varieties in $G$ such that
$$ 0 < \dim(V), \dim(W) < \dim(G).$$
Then at least one of the following holds:
\begin{itemize}
\item \textup{($G$ is not almost simple)} $G$ contains a proper normal algebraic subgroup $H$ of positive dimension.
\item \textup{(Escape)} For generic $a \in G$, there exists an essential product $(V^a \cdot W)^{\ess}$ of $V^a := a^{-1} Va$ and $W$ of dimension strictly greater than $\dim(W)$.
\end{itemize}
\end{lemma}

\begin{proof}  Suppose that $G$ is \emph{almost simple}, that is to say it contains no normal subgroups $H$ of positive dimension.

We can express $V$ and $W$ as unions of Zariski-open subsets of irreducible subvarieties of $G$.  By restricting to just one such top-dimensional subvariety for both $V$ and $W$, we may assume that the Zariski closures $\overline{V}, \overline{W}$ of $V, W$ in $G$ are irreducible.

Suppose that $V^a$ is such that $(V^a \cdot W)^{\ess}$ has dimension less than or equal to $\dim(W)$.
As $W$ is irreducible, the translates $v^a W$ for $v^a \in V^a$ are also irreducible with dimension $\dim(W)$.  Suppose that the set $\Sigma := \{ v^a W: v^a \in V^a \}$ of such translates is infinite.  Then any essential product $(V^a \cdot W)^{\ess}$ must contain a Zariski-dense subset of all but finitely many of such translates, and so has dimension strictly greater than $\dim(W)$, contradiction.  Thus $\Sigma$ is finite.  For each $W'$ in $\Sigma$, the set $\{ v^a \in V^a: v^a W = W' \}$ is easily seen\footnote{Here we use the Noetherian condition that there does not exist an infinite descending chain of closed varieties.} to be a constructible subset of $V^a$, and the union is all of $V^a$.  Thus there exists $W' \in \Sigma$ such that $\{ v^a \in V^a: v^a W = W' \}$ is Zariski-dense in $V$, i.e. $v^a W = W'$ for generic $v \in V$.  In particular, this implies that $(v^{-1} v')^a W = W$ for generic $v, v' \in V$.

The set $S := \{ g \in G: gW = W \}$ is an intersection of closed varieties $W w^{-1}$, $w \in W$ in $G$ and is thus closed (by the Noetherian condition); it is clearly a group, and is thus an algebraic subgroup of $G$.  Since each of the varieties $W w^{-1}$ has dimension $\dim(W)$, $S$ has dimension at most $\dim(W)$.

From the previous discussion we see that if $(V^a \cdot W)^{\ess}$ has dimension less than or equal to $\dim(W)$, then $(v^{-1} v')^a \in S$ for generic $v, v' \in V$.  In particular $S$ cannot have dimension $0$.

If the escape property fails, we conclude that $(v^{-1} v')^a \in S$ for generic $v, v' \in V$ and $a \in G$.  Call an $a$ \emph{good} if we have $v^{-1} v' \in S^{a^{-1}}$ for generic $v, v' \in V$. Then a generic element of $G$ is good.
Define a \emph{good set} to be an intersection of finitely many $S^{a^{-1}}$, such that $a$ is good.  By construction, all good sets $H$ are closed algebraic subgroups of $G$, and they have the property that $v^{-1} v' \in H$ for generic $v, v' \in V$.

By the Noetherian condition, there exists a good set $H$ which is minimal with respect to set inclusion.  Observe that $H^a$ is also good for generic $a$, and hence by minimality $H = H^a$ for generic $a$.  In other words, the normaliser $N(H)$ is Zariski dense in $G$.  But $N(H)$ is also a closed variety, and so $N(H) = G$, thus $H$ is normal in $G$.  By the almost simplicity of $G$, this forces $H$ to be zero-dimensional.  But this contradicts the fact that $v^{-1} v' \in H$ for generic $v, v' \in V$, since $\dim(V) > 0$, and the claim follows.
\end{proof}



\section{A lemma of Malcev and Platonov}\label{solv-app}

In this short appendix we sketch a brief proof of Lemma \ref{virtually-solvable}, whose statement we recall now.

\begin{virtually-solvable-repeat}
Any virtually solvable subgroup of $\GL_n(\C)$ \textup{(}that is to say, any subgroup of $\GL_n(\C)$ with a solvable subgroup of finite index\textup{)} contains a normal subgroup of index $O_n(1)$ which is simultaneously triangularisable, hence solvable.
\end{virtually-solvable-repeat}
\begin{proof}
Let $G$ be the Zariski closure of $\Gamma$. By inducting on $n$ we may assume that $G$, together with all of its subgroups of index at most $n!$, acts irreducibly on $\C^n$. Let $G_0$ be the (solvable) connected component of the identity of $G$ and let $U$ be its unipotent radical. If $U$ is non-trivial, the subspace $V$ of $\C^n$ consisting of all $x$ for which $ux = x$ for all $u \in U$ must be non-trivial. Since $G$ normalizes $U$ we have $g^{-1}ugx = x$ for all $g \in G$, which implies that $gV \subseteq V$. This, of course, contradicts the irreducibility assumption. It follows that $U$ is trivial and thus $G_0$ is a torus which, after a change of basis, is diagonal. Note that $G$ is contained in the normalizer $N(G_0)$ of $G_0$ inside $\GL_n(\C)$. The centralizer $Z(G_0)$ of $G_0$ in $\GL_n(\C)$ is easy to compute explicitly: it is a block diagonal subgroup. Moreover $N(G_0)$ permutes these diagonal blocks, and hence $|N(G_0)/Z(G_0)| \leq n!$. Passing to a subgroup of $G$ of index at most $n!$ we may thus assume that $G$ centralizes $G_0$. By assumption $G$ acts irreducibly: this forces $G_0$ to be trivial as $G$ fixes the weight spaces of $G_0$. It follows that $G_0$ is trivial and hence that $G$ is finite. An appeal to Jordan's theorem, which states that $G$ then has an \emph{abelian}, hence simultaneously diagonalisable, subgroup of index $O_n(1)$, concludes the argument.
\end{proof}

\section{Generating finite index subgroups}

The purpose of this appendix is to establish the following well-known lemma (which appears for instance as \cite[Lemma 4.8]{shalom} or as \cite[Lemma 6.7]{breuillard-tits}), required towards the end of the proof of Theorem \ref{mainthm3}.

\begin{lemma}[Bounded index subgroups have bounded length generators]\label{finite-index} Let $\Gamma$ be a group, and $S$ a finite symmetric generating set containing the identity. Let $\Gamma_0$ a subgroup of index $d$. Then $S^{2d-1}$ contains a generating set for $\Gamma_0$.
\end{lemma}	

\begin{proof} By the pigeonhole principle there is some $n \leq d-1$ such that $|S^n \Gamma_0/\Gamma_0| = |S^{n+1}\Gamma_0/\Gamma_0|$. For this $n$ we have $S^n\Gamma_0 = S^{n+1}\Gamma_0$. In particular, $S^n\Gamma_0$ is invariant under left multiplication by $S$ and is therefore the whole of $\Gamma$. It follows, of course, that $S^{d-1}\Gamma_0 = \Gamma$. Let $\gamma_i \in S^{d-1}$ for $i=1,\ldots,d$ be a full set of coset representatives of $\Gamma_0$ inside $\Gamma$ with $\gamma_1 \in \Gamma_0$. For each $i = 1,\dots, d$ and each $s \in S$ let $j = j(i,s) \in \{1,\dots, d\}$ be that index for which $s\gamma_i\Gamma_0=\gamma_{j}\Gamma_0$. Then it is straightforward to verify that the elements $\gamma_{j}^{-1}s\gamma_i\in \Gamma_0 \cap S^{2d-1}$ together with $\gamma_1$ form a generating set of $\Gamma_0$; just rewrite any word with letters in $S$ in terms of those elements.
\end{proof}

\setcounter{tocdepth}{1}

\end{document}